\documentclass[10pt]{amsart}

\usepackage{fullpage, palatino}
\usepackage{amsthm, amsfonts, amsmath}

\usepackage{pifont, tikz, subfigure}

\usetikzlibrary{plotmarks,shapes}
\long\def\greybox#1{%
    \newbox\contentbox%
    \newbox\bkgdbox%
    \setbox\contentbox\hbox to \hsize{%
        \vtop{
            \kern\columnsep
            \hbox to \hsize{%
                \kern\columnsep%
                \advance\hsize by -2\columnsep%
                \setlength{\textwidth}{\hsize}%
                \vbox{
                    \parindent=0bp
                    #1
                }%
                \kern\columnsep%
            }%
            \kern\columnsep%
        }%
    }%
    \setbox\bkgdbox\vbox{
        \pdfliteral{0.9 0.9 0.9 rg}
        \hrule width  \wd\contentbox %
               height \ht\contentbox %
               depth  \dp\contentbox
        \pdfliteral{0 0 0 rg}
    }%
    \wd\bkgdbox=0bp%
    \vbox{\hbox to \hsize{\box\bkgdbox\box\contentbox}}%
}

\newtheorem{theorem}{Theorem}
\newtheorem{lemma}[theorem]{Lemma}

\newtheorem{cor}[theorem]{Corollary}
\newtheorem{prop}[theorem]{Proposition}

\newcommand{\Sum}{\displaystyle\sum}

\newcommand{\y}[1]{\overline{Y}_{#1}}
\newcommand{\yq}[2]{\overline{Y}_{#1}(#2)}

\makeatletter
\def\testb#1{\testb@i#1,,\@nil}%
\def\testb@i#1,#2,#3\@nil{%
  \draw[->, thick] (O) --++(#1);
  \ifx\relax#2\relax\else\testb@i#2,#3\@nil\fi}
\makeatother   
\newcommand{\makediag}[1]{
    \coordinate (O) at (0,0); \coordinate (N) at (0,0.8);
    \coordinate (NE) at (0.8,0.8); \coordinate (E) at (0.8,0);
    \coordinate (SE) at (0.8,-0.8); \coordinate (S) at (0,-0.8);
    \coordinate (SW) at (-0.8,-0.8);\coordinate (W) at (-0.8,0);
    \coordinate (NW) at (-0.8,0.8); \coordinate (B1) at (1.2,1.2);
    \coordinate (B2) at (-1.2,-1.2);
    \testb{#1}
} 
\newcommand{\diagr}[1]{
  \begin{tikzpicture}[scale=0.3]\makediag{#1}\end{tikzpicture}
}

\newcommand{\mA}{\ensuremath{\mathcal{A}}}
\newcommand{\mB}{\ensuremath{\mathcal{B}}}
\newcommand{\mC}{\ensuremath{\mathcal{C}}}
\newcommand{\mD}{\ensuremath{\mathcal{D}}}
\newcommand{\mE}{\ensuremath{\mathcal{E}}}
\newcommand{\mS}{\ensuremath{\mathcal{S}}}
\newcommand{\pointplot}[1]{
    {\tiny
    \begin{tikzpicture}[scale=1.5]
		\draw [step=0.5,thin,gray!40] (-1.4,-1.4) grid (1.4,1.4);
		
		\draw [-] (-1.5,0) -- (1.5,0); 
		\draw [-] (0,-1.5) -- (0,1.5); 

		\draw [dashed,gray] (0,0) circle (1);

		\foreach \y in #1
		\node at \y [black,fill,circle,inner sep=0.8pt]  {};

		\foreach \x in {-1,-0.5,0.5,1}
  		\draw (\x,1pt) -- (\x,-3pt)
  		node[anchor=west] {\x};
		\foreach \y in {-1,-0.5,0.5,1}
  		\draw (1pt,\y) -- (-3pt,\y) ;

		\end{tikzpicture}
}
}

\newcommand{\pointplotT}[1]{
    \begin{tikzpicture}[scale=1.5]
    \draw [step=0.5,thin,gray!40] (-1.4,-0.6) grid (1.4,0.6);
    
    \draw [-] (-2,0) -- (2,0); 
    \draw [-] (0,-0.1) -- (0,0.1); 

    \draw [dashed,gray] (0,0) circle (0.5);

    \foreach \y in #1
    \node at \y [black,fill,circle,inner sep=0.8pt]  {};

    \foreach \x in {-0.5}
      \draw (\x,1pt) -- (\x,-3pt);
    \foreach \x in {0.5}
      \draw (\x,1pt) -- (\x,-3pt)
      node[anchor=north east] {\x};
    \foreach \y in {-0.5}
      \draw (1pt,\y) -- (-3pt,\y)
      node[anchor=north east] {\y};
    \foreach \y in {0.5}
      \draw (1pt,\y) -- (-3pt,\y)
      node[anchor=south east] {\y};
      
    \end{tikzpicture}
}

\newcommand{\aYpos}{(0.000000,-1.045090), (0.000000,1.045090), (-0.995454,-0.157278), (-0.995454,0.157278), (-0.965826,-0.313079), (-0.965826,0.313079), (-0.911199,-0.463257), (-0.911199,0.463257), (-0.832492,-0.603603), (-0.832492,0.603603), (-0.731465,-0.730075), (-0.731465,0.730075), (-0.610646,-0.838988), (-0.610646,0.838988), (-0.473210,-0.927165), (-0.473210,0.927165), (-0.322854,-0.992040), (-0.322854,0.992040), (-0.163667,-1.031731), (-0.163667,1.031731), (0.163667,-1.031731), (0.163667,1.031731), (0.322854,-0.992040), (0.322854,0.992040), (0.473210,-0.927165), (0.473210,0.927165), (0.610646,-0.838988), (0.610646,0.838988), (0.731465,-0.730075), (0.731465,0.730075), (0.832492,-0.603603), (0.832492,0.603603), (0.911199,-0.463257), (0.911199,0.463257), (0.965826,-0.313079), (0.965826,0.313079), (0.995454,-0.157278), (0.995454,0.157278)}

\newcommand{\bYpos}{(-0.990718,-0.135936), (-0.990718,0.135936), (-0.960249,-0.279143), (-0.960249,0.279143), (-0.922011,-0.387164), (-0.922011,0.387164), (-0.881366,-0.551551), (-0.881366,0.551551), (-0.812139,-0.665027), (-0.812139,0.665027), (-0.675248,-0.822584), (-0.675248,0.822584), (-0.580551,-0.892900), (-0.580551,0.892900), (-0.381520,-1.003236), (-0.381520,1.003236), (-0.284095,-1.032146), (-0.284095,1.032146), (-0.042675,-1.073245), (-0.042675,1.073245), (0.041777,-1.068741), (0.041777,1.068741), (0.298932,-1.025779), (0.298932,1.025779), (0.361040,-1.000152), (0.361040,1.000152), (0.602540,-0.868771), (0.602540,0.868771), (0.639741,-0.835693), (0.639741,0.835693), (0.833804,-0.623725), (0.833804,0.623725), (0.849787,-0.595279), (0.849787,0.595279), (0.969451,-0.322178), (0.969451,0.322178), (0.972401,-0.306467), (0.972401,0.306467)}

\newcommand{\cYpos}{(0.000246,-1.042826), (0.000246,1.042826), (0.962892,-0.312373), (0.962892,0.312373), (-1.001192,-0.157410), (-1.001192,0.157410), (-0.967813,-0.313551), (-0.967813,0.313551), (-0.926982,-0.469078), (-0.926982,0.469078), (-0.834491,-0.605005), (-0.834491,0.605005), (-0.752877,-0.747432), (-0.752877,0.747432), (-0.611342,-0.840205), (-0.611342,0.840205), (-0.494736,-0.960513), (-0.494736,0.960513), (-0.322499,-0.991780), (-0.322499,0.991780), (-0.178885,-1.085792), (-0.178885,1.085792), (0.162014,-1.114978), (0.162014,1.114978), (0.321763,-0.988265), (0.321763,0.988265), (0.481018,-1.059617), (0.481018,1.059617), (0.607701,-0.835038), (0.607701,0.835038), (0.713324,-0.897753), (0.713324,0.897753), (0.828508,-0.600997), (0.828508,0.600997), (0.888618,-0.626741), (0.888618,0.626741), (0.991201,-0.316281), (0.991201,0.316281)}

\newcommand{\aTpts}{(-1.792646,-0.475809), (-1.792646,0.475809), (-1.564355,-0.270893), (-1.564355,0.270893), (-1.502164,-0.388816), (-1.502164,0.388816), (-1.363577,-0.229962), (-1.363577,0.229962), (-1.214176,-0.299700), (-1.214176,0.299700), (-1.164466,-0.188088), (-1.164466,0.188088), (-1.131077,-0.135341), (-1.131077,0.135341), (-0.988324,-0.089109), (-0.988324,0.089109), (-0.981886,-0.110456), (-0.981886,0.110456), (-0.968285,-0.144650), (-0.968285,0.144650), (-0.931603,-0.206332), (-0.931603,0.206332), (-0.870205,-0.072438), (-0.870205,0.072438), (-0.835856,-0.084399), (-0.835856,0.084399), (-0.795809,-0.050983), (-0.795809,0.050983), (-0.778091,-0.098368), (-0.778091,0.098368), (-0.776887,-0.321797), (-0.776887,0.321797), (-0.755352,-0.054881), (-0.755352,0.054881), (-0.743183,-0.037683), (-0.743183,0.037683), (-0.702097,-0.038263), (-0.702097,0.038263), (-0.695977,-0.056404), (-0.695977,0.056404), (-0.665457,-0.102785), (-0.665457,0.102785), (-0.664694,-0.028007), (-0.664694,0.028007), (-0.646613,-0.035989), (-0.646613,0.035989), (-0.637237,-0.021251), (-0.637237,0.021251), (-0.616386,-0.016580), (-0.616386,0.016580), (-0.614351,-0.024623), (-0.614351,0.024623), (-0.603843,-0.046657), (-0.603843,0.046657), (-0.591967,-0.017700), (-0.591967,0.017700), (-0.575728,-0.013203), (-0.575728,0.013203), (-0.571026,-0.025449), (-0.571026,0.025449), (-0.563533,-0.010139), (-0.563533,0.010139), (-0.554121,-0.007970), (-0.554121,0.007970), (-0.551674,-0.015509), (-0.551674,0.015509), (-0.546692,-0.006386), (-0.546692,0.006386), (-0.539312,-0.010185), (-0.539312,0.010185), (-0.530927,-0.007061), (-0.530927,0.007061), (-0.524975,-0.005102), (-0.524975,0.005102), (-0.520595,-0.003809), (-0.520595,0.003809), (-0.517276,-0.002920), (-0.517276,0.002920), (-0.514702,-0.002289), (-0.514702,0.002289), (-0.512664,-0.001828), (-0.512664,0.001828), (-0.511023,-0.001483), (-0.511023,0.001483), (0.511023,-0.001483), (0.511023,0.001483), (0.512664,-0.001828), (0.512664,0.001828), (0.514702,-0.002289), (0.514702,0.002289), (0.517276,-0.002920), (0.517276,0.002920), (0.520595,-0.003809), (0.520595,0.003809), (0.524975,-0.005102), (0.524975,0.005102), (0.530927,-0.007061), (0.530927,0.007061), (0.539312,-0.010185), (0.539312,0.010185), (0.546692,-0.006386), (0.546692,0.006386), (0.551674,-0.015509), (0.551674,0.015509), (0.554121,-0.007970), (0.554121,0.007970), (0.563533,-0.010139), (0.563533,0.010139), (0.571026,-0.025449), (0.571026,0.025449), (0.575728,-0.013203), (0.575728,0.013203), (0.591967,-0.017700), (0.591967,0.017700), (0.603843,-0.046657), (0.603843,0.046657), (0.614351,-0.024623), (0.614351,0.024623), (0.616386,-0.016580), (0.616386,0.016580), (0.637237,-0.021251), (0.637237,0.021251), (0.646613,-0.035989), (0.646613,0.035989), (0.664694,-0.028007), (0.664694,0.028007), (0.665457,-0.102785), (0.665457,0.102785), (0.695977,-0.056404), (0.695977,0.056404), (0.702097,-0.038263), (0.702097,0.038263), (0.743183,-0.037683), (0.743183,0.037683), (0.755352,-0.054881), (0.755352,0.054881), (0.776887,-0.321797), (0.776887,0.321797), (0.778091,-0.098368), (0.778091,0.098368), (0.795809,-0.050983), (0.795809,0.050983), (0.835856,-0.084399), (0.835856,0.084399), (0.870205,-0.072438), (0.870205,0.072438), (0.931603,-0.206332), (0.931603,0.206332), (0.968285,-0.144650), (0.968285,0.144650), (0.981886,-0.110456), (0.981886,0.110456), (0.988324,-0.089109), (0.988324,0.089109), (1.131077,-0.135341), (1.131077,0.135341), (1.164466,-0.188088), (1.164466,0.188088), (1.214176,-0.299700), (1.214176,0.299700), (1.363577,-0.229962), (1.363577,0.229962), (1.502164,-0.388816), (1.502164,0.388816), (1.564355,-0.270893), (1.564355,0.270893), (1.792646,-0.475809), (1.792646,0.475809)}

\newcommand{\bTpts}{(-1.102612,-0.356765), (-1.102612,0.356765), (-1.101997,-0.251891), (-1.101997,0.251891), (-1.092032,-0.192632), (-1.092032,0.192632), (-0.868038,-0.111907), (-0.868038,0.111907), (-0.834149,-0.128396), (-0.834149,0.128396), (-0.779048,-0.147243), (-0.779048,0.147243), (-0.715408,-0.063618), (-0.715408,0.063618), (-0.677958,-0.157780), (-0.677958,0.157780), (-0.665298,-0.061524), (-0.665298,0.061524), (-0.652191,-0.036453), (-0.652191,0.036453), (-0.610282,-0.030897), (-0.610282,0.030897), (-0.597093,-0.047948), (-0.597093,0.047948), (-0.579312,-0.016495), (-0.579312,0.016495), (-0.560451,-0.016338), (-0.560451,0.016338), (-0.546631,-0.000000), (-0.546631,0.000000), (-0.543689,-0.000000), (-0.543689,0.000000), (-0.538316,-0.000000), (-0.538316,0.000000), (-0.530747,-0.000000), (-0.530747,0.000000), (-0.529629,-0.000000), (-0.529629,0.000000), (-0.520318,-0.000000), (-0.520318,0.000000), (-0.514646,-0.000000), (-0.514646,0.000000), (-0.511022,-0.000000), (-0.511022,0.000000), (-0.508584,-0.000000), (-0.508584,0.000000), (0.543078,-0.007718), (0.543078,0.007718), (0.549565,-0.009565), (0.549565,0.009565), (0.557667,-0.012067), (0.557667,0.012067), (0.567982,-0.015553), (0.567982,0.015553), (0.581415,-0.020577), (0.581415,0.020577), (0.599395,-0.028127), (0.599395,0.028127), (0.624285,-0.040106), (0.624285,0.040106), (0.660205,-0.060534), (0.660205,0.060534), (0.714694,-0.099070), (0.714694,0.099070), (0.718646,-0.046049), (0.718646,0.046049), (0.761804,-0.061377), (0.761804,0.061377), (0.801280,-0.183399), (0.801280,0.183399), (0.820618,-0.085254), (0.820618,0.085254), (0.904298,-0.125324), (0.904298,0.125324), (0.927958,-0.413327), (0.927958,0.413327), (1.029963,-0.200026), (1.029963,0.200026), (1.229972,-0.362815), (1.229972,0.362815), (1.360404,-0.293904), (1.360404,0.293904)}

\newcommand{\cTpts}{(-1.507645,-0.388545), (-1.507645,0.388545), (-1.401604,-0.477921), (-1.401604,0.477921), (-1.365585,-0.230394), (-1.365585,0.230394), (-1.069349,-0.241910), (-1.069349,0.241910), (-1.058006,-0.318612), (-1.058006,0.318612), (-1.002789,-0.443909), (-1.002789,0.443909), (-0.988497,-0.091085), (-0.988497,0.091085), (-0.982140,-0.112993), (-0.982140,0.112993), (-0.968676,-0.148218), (-0.968676,0.148218), (-0.932126,-0.212381), (-0.932126,0.212381), (-0.844007,-0.127735), (-0.844007,0.127735), (-0.799307,-0.143598), (-0.799307,0.143598), (-0.795678,-0.053210), (-0.795678,0.053210), (-0.773719,-0.337964), (-0.773719,0.337964), (-0.755078,-0.057640), (-0.755078,0.057640), (-0.733231,-0.066423), (-0.733231,0.066423), (-0.725873,-0.155789), (-0.725873,0.155789), (-0.695338,-0.059895), (-0.695338,0.059895), (-0.690432,-0.066100), (-0.690432,0.066100), (-0.664444,-0.029963), (-0.664444,0.029963), (-0.633544,-0.060184), (-0.633544,0.060184), (-0.632170,-0.036756), (-0.632170,0.036756), (-0.616201,-0.017985), (-0.616201,0.017985), (-0.613934,-0.026712), (-0.613934,0.026712), (-0.602154,-0.050628), (-0.602154,0.050628), (-0.598326,-0.135000), (-0.598326,0.135000), (-0.587012,-0.029999), (-0.587012,0.029999), (-0.575473,-0.014534), (-0.575473,0.014534), (-0.561192,-0.017282), (-0.561192,0.017282), (-0.559248,-0.039870), (-0.559248,0.039870), (-0.553962,-0.008862), (-0.553962,0.008862), (-0.551047,-0.017203), (-0.551047,0.017203), (-0.545428,-0.010923), (-0.545428,0.010923), (-0.535714,-0.016685), (-0.535714,0.016685), (-0.530654,-0.007932), (-0.530654,0.007932), (-0.523513,-0.008523), (-0.523513,0.008523), (-0.520460,-0.004312), (-0.520460,0.004312), (-0.516572,-0.004930), (-0.516572,0.004930), (-0.514629,-0.002604), (-0.514629,0.002604), (-0.512283,-0.003105), (-0.512283,0.003105), (-0.510980,-0.001692), (-0.510980,0.001692), (-0.250000,-0.433013), (-0.250000,0.433013), (0.528634,-0.459518), (0.528634,0.459518), (0.542366,-0.016347), (0.542366,0.016347), (0.546892,-0.005201), (0.546892,0.005201), (0.548177,-0.020332), (0.548177,0.020332), (0.554383,-0.006513), (0.554383,0.006513), (0.555131,-0.025707), (0.555131,0.025707), (0.563456,-0.033104), (0.563456,0.033104), (0.563882,-0.008322), (0.563882,0.008322), (0.573348,-0.043488), (0.573348,0.043488), (0.576206,-0.010893), (0.576206,0.010893), (0.584852,-0.058337), (0.584852,0.058337), (0.592642,-0.014695), (0.592642,0.014695), (0.597588,-0.079840), (0.597588,0.079840), (0.603566,-0.316753), (0.603566,0.316753), (0.610292,-0.111078), (0.610292,0.111078), (0.615340,-0.020605), (0.615340,0.020605), (0.620185,-0.156215), (0.620185,0.156215), (0.621924,-0.221284), (0.621924,0.221284), (0.648141,-0.030418), (0.648141,0.030418), (0.688519,-0.104503), (0.688519,0.104503), (0.698526,-0.048291), (0.698526,0.048291), (0.711329,-0.133775), (0.711329,0.133775), (0.736145,-0.172238), (0.736145,0.172238), (0.744064,-0.032399), (0.744064,0.032399), (0.763153,-0.222879), (0.763153,0.222879), (0.782907,-0.085688), (0.782907,0.085688), (0.793166,-0.291202), (0.793166,0.291202), (0.797055,-0.044219), (0.797055,0.044219), (0.826939,-0.388983), (0.826939,0.388983), (0.872082,-0.063473), (0.872082,0.063473), (0.943078,-0.184280), (0.943078,0.184280), (0.975192,-0.382773), (0.975192,0.382773), (0.984990,-0.097967), (0.984990,0.097967), (1.170447,-0.169316), (1.170447,0.169316), (1.517350,-0.356979), (1.517350,0.356979), (1.571559,-0.247700), (1.571559,0.247700)}

\newcommand{\EM}[1]{}

\begin{document}
\title{Singularity analysis via the iterated kernel method}
\author{Stephen Melczer \and Marni Mishna}
\address{Department of Mathematics\\ Simon Fraser University\\
Burnaby, Canada, V5A 1S6}
\thanks{This work was supported by NSERC via a Discovery Grant, a
    USRA fellowship, and a Michael Smith Foreign Study Supplement. SM
    was also partially supported by funding from the Office for Science and
    Technology at the Embassy of France in Canada.
}

\begin{abstract}
  In the quarter plane, five lattice path models with unit steps have
  resisted the otherwise general approach featured in recent
  works by Fayolle, Kurkova and Raschel. Here we consider these five
  models, called the singular models, and prove that the univariate
  generating functions marking the number of walks of a given length
  are not D-finite.  Furthermore, we provide exact and asymptotic
  enumerative formulas for the number of such walks, and describe an
  efficient algorithm for exact enumeration.
\end{abstract}

\maketitle

\centerline{\em Dedicated to the remarkable Philippe Flajolet.}

\section{Introduction}
Lattice path models are classical objects, appearing very naturally in
a variety of probabilistic and combinatorial contexts. Recent work has
shown how they can help us better understand generating functions in a more general
setting by addressing the question of predicting when the generating
function of a combinatorial class will satisfy a `nice' differential
equation.  Lattice path models restricted to the quarter plane have
proved to be very useful in this regard -- they offer a family of
generating functions which are straightforward to manipulate, yet
which possess some surprising structure.  This, in turn, has led to
some useful innovations in enumeration, including applications of
boundary value methods~\cite{FaIaMa99, FaRa10, KuRa11}, powerful and
widely applicable variants of the kernel method~\cite{Bous05, BoMi10,
  MiRe09}, original computer algebra approaches~\cite{BoKa09,
  KaKoZe09}, and some freshened restatements of classic number theory
results~\cite{BoRaSa13}.

As has been remarked upon previously~\cite{Chri90,FlGeSa04}, the
property of being D-finite, that is, of satisfying a linear
differential equation with polynomial coefficients, is exceptional and
not expected of an arbitrary function. Indeed, in the case of
combinatorial generating functions, the property of being D-finite
appears to correlate with rich structure in the corresponding class --
structure which we have yet to fully uncover.  What can we learn from
lattice path models?

A key observation of Bousquet-M\'elou and Mishna \cite{BoMi10} was that lattice path
models with small steps restricted to the quarter plane appeared to be
naturally partitioned according to the nature of their generating
functions: specifically, they observed that D-finiteness of a model's
generating function appeared to be correlated to the finiteness of a
group of plane transformations derived from the set of
allowable steps. Furthermore, in two dimensions this property is
further correlated with more combinatorial qualities of the step set:
for example, symmetry across an axis or rotational symmetry, but
\emph{not} $x\leftrightarrow y$ symmetry.  In some cases the
explanation is well understood, such as Theorem~1
in~\cite{BoPe03}. When the drift (that is, the vector sum of the
allowable directions) is zero, Fayolle and Raschel~\cite{FaRa11}
describe an arithmetic condition which begs a more combinatorial
interpretation. Of the 79 non-isomorphic models, 23 are well studied
with D-finite generating functions and 51 are highly suspected to be
non D-finite: Kurkova and Raschel~\cite{KuRa11} proved that the
trivariate generating functions marking endpoint are not D-finite by solving 
related boundary value problems, and
Bostan, Raschel and Salvy~\cite{BoRaSa13} proved the excursion (walks
returning to the origin) generating functions are not D-finite via an
argument on the asymptotics of the coefficients. The remaining five
models are called \emph{singular\/} and resist both these strategies
for different reasons. (For example, the excursion generating function
is trivally 1 in these cases). Two of these models were previously
considered~\cite{MiRe09}, and both (univariate) generating functions
were proven to be non D-finite. We apply this strategy to the final three
models: it is an application of the
\emph{iterated kernel method} inspired by Bousquet-M\'elou and
Petkov\v sek~\cite{BoPe03}, and Janse van Rensburg, Prellberg and
Rechnitzer~\cite{JaPrRe08}.

Specifically, the present work proves that the three remaining cases are
not D-finite by analyzing an explicit generating
function -- asymptotic analysis and rapid exact enumeration are
applications of this expression.  In each case, we show that the number
of singularities is infinite and far enough away from the dominating
pole that they do not affect the first order asymptotics.  The
(essentially technical) challenge, as was the case in~\cite{MiRe09},
is the justification that these singularities are true poles and are
not somehow canceled by a quirk of the expression. This is significant
because a D-finite function has a finite number of singularities, and
so such a demonstration is a proof of non D-finiteness of the
generating function. We diverge from~\cite{MiRe09} slightly and use a
parameterization with its origins in the method of~\cite{FaIaMa99}, as
this yields a simpler process.  In the course of our proofs we revisit
some older theorems on polynomials that a reader faced with a similar
problem may find useful. For a more general combinatorial discussion
about why these models have D-finite generating functions, we direct
the reader to~\cite{MiRe09}.

In summary, for each of the five singular models we take a unified
approach to prove formulas for asymptotic enumeration and determine
an explicit expression for the generating function, information which 
cannot be determined using other known methods.  In addition, we
prove that the (univariate) counting generating function is not
D-finite for the five models.

\subsection{The family of singular models}
A lattice path model is defined by a set of vectors -- the allowable
directions in which one can move along the sublattice $\mathbb{N}^2$~of~$\mathbb{Z}^2$. We are initially interested in models which permit only
``small'' steps; that is, the steps are contained in $\{0, +1,
-1\}^2$.  We use the notation ${\sf NW}\equiv(-1,1),{\sf N}\equiv(0,1),
{\sf NE}\equiv(1,1)$, etc. The family of singular models consists of the
following five models:

\begin{center}\small
\greybox{\begin{tabular}{lllllllll}
$\mA=$ \diagr{NW,NE,SE}$\/=\{ \sf NW, NE, SE \}$&
$\mB=$ \diagr{NW,N,E,SE}$\/=\{ \sf NW, N, E, SE \}$&
$\mC=$ \diagr{NW,N,NE,E,SE}$\/=\{ \sf NW, N, NE, E, SE \}$\\[2mm]
$\mD=$ \diagr{NW,N,SE} $\/=\{ \sf NW, N, SE \}$& 
$\mE=$ \diagr{NW,N,NE,SE}$\/=\{ \sf NW, N, NE, SE \}$&
\end{tabular}}
\end{center}

Models~$\mA$ and~$\mD$ are the two models considered by Mishna and
Rechnitzer, and their strategy, known as the iterated kernel method,  extends to all of these models. Note that
the present work corrects an analytical error found
in~\cite{MiRe09}, which does not substantially change the stated
results.

For each model $\mS\in\{\mA,\mB,\mC,\mD,\mE\}$ we address the following:
\begin{itemize} \setlength{\itemsep}{0mm} \sl
\item[\ding{172}] What is the number~$S_n$ of walks of length~$n$ beginning at the origin and staying in $\mathbb{N}^2$?
\item[\ding{173}] How does~$S_n$ grow asymptotically when~$n$ is large?
\item[\ding{174}] Is the generating function $S(t)=\sum_n S_n t^n$ D-finite? 
\end{itemize}
The approach is to give an explicit expression for the generating
function via the iterated kernel method, which entails describing
a functional equation for a multivariate generating function,
isolating its kernel. We generate a telescoping sum using a prescribed
sequence of pairs which annihilate the kernel. The derived expression
is useful to deduce asymptotic information and to demonstrate the
source of an infinite collection of singularities. 

The next section describes how to obtain generating function
expressions. This is followed by the asymptotic analysis and non
D-finiteness proofs for the symmetric models, and we conclude with a
summary of the analysis of the asymmetric models.

\begin{table}\center
\begin{tabular}{c|ll}
Model & First 10 Terms in the counting sequence &(OEIS Tag) \\ \hline
\mA & $1, 1, 3, 7, 21, 55, 165, 457, 1371, 3909, 11727$ &(A151267) \\
\mB &  $1, 2, 6, 20, 70, 254, 942, 3550, 13532, 52030, 201386$ &(A151284)\\
\mC & $1, 3, 13, 59, 279, 1341, 6527, 31995, 157659, 779601, 3864985$ &(A151321) \\
\mD & $1, 1, 2, 4, 10, 23, 61, 153, 418, 1100, 3064$ &(A151256)  \\
\mE & $1, 2, 7, 24, 91, 339, 1316, 5064, 19876, 77655, 306653$
&(A151294) \\[2mm]
\end{tabular}\\
\label{tab:terms}
\caption{The initial terms in the counting sequences for the number of
walks of a given length with steps from the given model, restricted to
the quarter plane. The OEIS Tag refers to the corresponding entry
in the Online Encyclopedia of Integer Sequences~\cite{OEIS}}
\end{table}

\section{An explicit expression for the generating function}
\subsection{The functional equation and its kernel}
Our central mathematical object is the multivariate
generating function $S_{x,y}(t)=\sum_{i,j,n}s_{ij}(n) x^iy^j t^n$, where
$s_{ij}(n)$ counts the number of walks of length $n$ with steps from $\mS$ which begin at the origin, end at the point $(i,j)$, and stay in $\mathbb{N}^2$. (Throughout, \mS~is our generic step set.) Our goal is to determine properties of $S(t)\equiv S_{1,1}(t)$, the generating function for the number of walks in the quarter plane.

To each of the five step sets, we associate a polynomial called the
\emph{kernel}; for the step set $\mS$, define
\[ K_S(x,y)=xy -txy \sum_{(i,j)\in \mS} x^iy^j.\] As we restrict
ourselves to small steps, the inventory of the steps has the form
\begin{equation}
 \sum_{(i,j)\in \mS} x^i y^j = xP_1(y) + P_0(y) + \frac{1}{x}
P_{-1}(y) = yQ_1(x) + Q_0(x) + \frac{1}{y} Q_{-1}(x). 
\end{equation}
Thus, $K_S(x,y)$ can be regarded as a quadratic in $y$ (respectively $x$)
whose coefficients contain $t$, $x$ and the $Q_i(x)$ (resp. $t$, $y$, and
$P_i(y)$):
\begin{equation}\label{eq:kernel-as-quadratic}
K_S(x,y)= -xtQ_1(x)\,y^2 + \left(x-txQ_0(x)\right)\,y - xtQ_{-1}(x).
\end{equation}
When the model is clear, we omit the subscript $\mS$. One common
property of the singular models is that they contain the steps ${\sf
  NW}$ and ${\sf SE}$, and at least one other step, which prevents degeneracy in the quadratic. 

Each model admits a functional equation for $S_{x,y}(t)$. We apply the common
decomposition that a walk is either the empty walk, or a shorter walk
followed by a single step. Taking into account the restrictions on walk location, as well as the fact that substituting $x=0$ (respectively $y=0$) into the function $S_{x,y}(t)$ gives the generating function of walks ending on the $y$-axis (respectively $x$-axis), we obtain, as many others have before us,  the functional equation 
\begin{equation}\label{eqn:fund}
K(x,y)S_{x,y}(t) = xy+K(x,0)S_{x,0}(t)+K(0,y)S_{0,y}(t). 
\end{equation}

We are interested in the solutions to the kernel equation of the form:
\begin{equation}
K(x, Y_+(x;t))=K(x, Y_-(x;t))=K(X_+(y;t), y)=K(X_-(y;t), y)=0,
\end{equation} 
and these algebraic functions are easily determined since the kernel is a quadratic:
\begin{align}
Y_{\pm}(x;t) &= \frac{ (1 - t Q_0(x)) \mp \sqrt{ \left(Q_0(x)^2-4Q_1(x)Q_{-1}(x)\right) t^2 -2Q_0(x)t + 1} }{2tQ_1(x)}    \\
X_{\pm}(y;t) &= \frac{ (1 - t P_0(y)) \mp \sqrt{ \left(P_0(y)^2-4P_1(y)P_{-1}(y)\right) t^2 -2P_0(y)t + 1} }{2tP_1(y)}.
\end{align}
There are other function pairs which annihilate the kernel, as we shall see. Remark that the boundary value method begins as we have, with the functional equation~\eqref{eqn:fund}, but ultimately uses a different parametrization for the roots of the kernel, and from there a very different means to get access to the generating function.

The generating function has a natural expression in terms of iterated compositions of the $Y$ and $X$, hence the name \emph{iterated kernel method}.

\subsection{Summary: The Symmetric Models} 

We define the sequence of functions $Y_i$ by the recurrence
$Y_{n+1}(x)=Y_n(Y_+(x;t))$ with base case $Y_0(x)=x$. Remark that
$Y_1\equiv Y_+$.  In Section~\ref{sec:symmodels} we show that if $\mS$
is a model symmetric about the line $x=y$, i.e. $\mS\in\{ \mA, \mB,
\mC\}$, then one has the explicit form 

\greybox{
\begin{equation}\label{eq:sym}
 S(t) = \frac{1}{1-|\mS|t}\left(1-2\sum_{n=0}^\infty (-1)^{n} Y_n(1)Y_{n+1}(1)\right).  
\end{equation}
}

Theorem~\ref{thm:dom-sing} gives the first order
asymptotics of the symmetric models, extracted from this
expression. Table~\ref{tab:polys} provides polynomial equations that are
satisfied by the remaining poles.

Section~\ref{sec:not-dfinite-sym} outlines the proof that these models
are not D-finite. For each model we provide an infinite family of
poles, which is sufficient because D-finite series have at most a
finite number of poles.  We prove that an infinite collection of $Y_n(1)$ contribute
singularities; the main difficulty lies in ensuring that the
singularities in the terms are genuinely singularities of the sum.

\subsection{Summary: The Asymmetric Models}
The remaining two models are treated similarly, but require a bit more work. Here we require two function sequences:
 \begin{equation*}
 \chi_n(x) = X_+ (Y_+ (\chi_{n-1}(x))), \quad \chi_0(x)=x \qquad\text{and}\qquad
 \Upsilon_n(y) = Y_+ (X_+(\Upsilon_{n-1}(x))),  \quad \Upsilon_0(y)=y.
\end{equation*}

We show in Section~\ref{sec:asymmodels} that the generating function for the asymmetric walks is

\greybox{
{\small \begin{equation}\label{eq:assym}
S(t) = \frac{1}{1-|\mS|t}\left(1
-t\Sum_{n\geq 0}\chi_n(1)\cdot\left(Y_+(\chi_n(1))-Y_+(\chi_{n-1}(1))\right)
-t\Sum_{n\geq 0}X_+(\Upsilon_n(1))\cdot\left( \Upsilon_n(1)-\Upsilon_{n+1}(1)\right)\right).
\end{equation}
}
}

Of course, this expression is also valid for the symmetric models. In
this case, $X_+=Y_+$ and expression~\eqref{eq:assym} reduces to equation~\eqref{eq:sym}

The asymptotics are considered in Section~\ref{sec:aasym}, and D-finiteness results are considered in Section~\ref{sec:anonD}.

\subsection{What makes this family special?}
Consider the lowest order terms of the roots of the
kernel as a power series in $t$. They are 
\[Y_+ = P_{-1}(x)t + O(t^2) \qquad \mbox{ and } \qquad Y_{-} =
\frac{1}{tP_1(x)}-\frac{P_0(x)}{P_1(x)} + O(t),\] where $P_r(x) =
\sum_{(i,r)\in\mS}x^i$.  Of the 56 (conjectured) non D-finite models
only 5 models, precisely the singular family we are studying,
have a lowest order term with a positive power in $x$ and $t$, implying that the infinite sum obtained by the
iterated kernel method converges.  This prevents the method from being
applied to a broader range of models in this context. 

\subsection{Fast enumeration} The focus of the present work is to
prove that the generating functions of singular models are not
D-finite. Nonetheless, we should not lose sight of the fact that our
generating function expressions are also useful for enumeration.

In fact, we can use the series expression to generate the
first $N$ terms of $S_{1,0}(t)$ and $S_{1,1}(t)$ for each model with
$\tilde O\left(N^3\right)$ bit-complexity (where the notation $\tilde
O(\cdot)$ suppresses logarithmic factors), which is an order of
magnitude faster than the $\tilde O\left(N^4\right)$ bit-complexity of
the naive generation algorithm.  The key fact which gives this speedup
is that we can form a linear recurrence for $1/Y_n$ (see
Table~\ref{tab:sym-rec}). To be more specific, we define $Z_n :=
1\left/ \left(Y_n/t^n\right) \right.$ and:
\begin{enumerate}
  \item Generate $Z_0$ and $Z_1$ to precision $2N$ 
  \item Use the recurrences in Table~\ref{tab:sym-rec} to form a linear recurrence for $Z_n$
  \item Exploit this recurrence to generate $Z_2,\dots, Z_{\lfloor \frac{N}{2} \rfloor}$ to precision $2N$ using only shifts and additions
  \item Recover $Y_n = t^n/Z_n \mod t^N$ for $n=1,\dots,\lfloor \frac{N}{2} \rfloor$.
\end{enumerate}
The series expressions then allows us to generate the first $N$ terms
of the generating function.  The cost of generating these terms is
dominated by the inversion of the $Z_n$, which have summands whose
bit-size grows linearly.  

Although the generating functions are not D-finite, and hence the
coefficients do not satisfy a fixed length linear recurrence with polynomial
coefficients, we are able to generate the terms in a more efficient
manner. In order to generate a large number of
terms, one may wish to use modular methods in order to prevent a
memory overflow.

\section{Symmetric models: $\mA, \mB, \mC$}
\label{sec:symmodels}

\subsection{An explicit generating function expression}
We focus first on the three models \mA, \mB~and \mC, as these
models are symmetric about the line $x=y$.  As such, these models benefit from
the relation $S_{x,0}=S_{0,x}$, and Equation~\eqref{eqn:fund} can be
rewritten as
\begin{equation}\label{eqn:fund-sym}
K(x,y)S_{x,y}(t) = xy+K(0,y)S_{0,y}(t)+K(0,x)S_{0,x}(t).  
\end{equation}

Our iterates satisfy $Y_{n+1}(x)=Y_+(Y_n(x)),\, Y_0(x)=x$, and hence we see
 \[K(Y_n, Y_{n+1})=K(Y_n(x), Y_+(Y_n(x)))=0\]
for all $n$ by substituting $x=Y_n(x)$ into the kernel relation $K(x, Y_+(x))=0$.  Thus, when we make this substitution into Equation~\eqref{eqn:fund-sym} we find for each $n$:
\[ 0 = Y_n(x)Y_{n+1}(x) + K(0,Y_{n+1}(x))S_{0,Y_{n+1}(x)}(t) + K(0,Y_n(x))S_{0,Y_n(x)}(t).\]
We can determine an expression for  $K(0,x)S_{0,x}(t)$ by taking an alternating sum of these equations since all of the $ K(0,Y_n(x))S_{0,Y_n(x)}(t)$ terms are canceled for $n>0$ in a telescoping sum:
\begin{align*}
0&=\sum_{n=0}^\infty (-1)^n \Bigl(Y_n(x)Y_{n+1}(x)
  +K(0,Y_{n+1}(x))S_{0,Y_{n+1}(x)}(t) + K(0,Y_n)S_{0,Y_n(x)}(t)\Bigr)\\
&=K(0, x) S_{0,x}(t)+ \sum_{n=0}^\infty (-1)^n Y_n(x)Y_{n+1}(x) .
\end{align*}
We rearrange this and evaluate at $x=1$ to express the counting generating function for walks returning to the axis:
\begin{equation}\label{eq:S01-sym}
S_{0,1}(t) = \frac{1}{t}\sum_{n=0}^\infty (-1)^{n}  Y_n(1)Y_{n+1}(1), 
\end{equation}
as $K(0, 1) = -t$ for each case considered here. This converges as a power series because in each of these cases $Y_n(x)=O(t^n)$. 

Furthermore, substituting $x=1$ and $y=1$ into Equation~\eqref{eqn:fund-sym} gives the full counting generating function
\begin{equation}\label{eq:S11-sym}
  S(t) = \frac{1-2tS_{0,1}(t)}{1-t|\mS|} = \frac{1}{1-t |S|} \left(1 - 2 \sum_n (-1)^{n}  Y_n(1)Y_{n+1}(1) \right).
\end{equation}
We address the robustness of this expression as a complex function in Theorem~\ref{thm:converge}, after 
we determine an explicit expression for $Y_n(1)$ as a rational function of $Y_1(1)$. 

\subsection{Asymptotic Enumeration}
Next we show
that the sum $\sum_n (-1)^{n} Y_n(1)Y_{n+1}(1)$ is convergent at
$t=\frac{1}{|\mS|}$, and that its radius of convergence is bounded
below by $t=\frac{1}{p_0+2\sqrt{p_1p_{-1}}}$ where $p_i =
P_i(1)=|\{(i,r): -1\leq r \leq 1, (i,r)\in\mS\}|$.  In our three symmetric models,
the singularity of $S(t)$ at $\frac{1}{|\mS|}$ is thus dominant.
\begin{theorem} \label{thm:dom-sing} 
For each model $\mS$ in $\{\mA, \mB, \mC\}$, the number of walks
  $S_n=[t^n]S(t)$ grows asymptotically like
\begin{equation}\label{eq:sym_asmpt}
S_n\sim \kappa_S \left(\frac{1}{|\mS|}\right)^n +O\left( \left(
    p_0+2\sqrt{p_1p_{-1}} \right)^n \right),
\end{equation}
  where each $\kappa_S$ is a constant which can be calculated to arbitrary precision using Equation~\eqref{eq:S01-sym}.
\end{theorem} 

\begin{proof}
  We proceed by basic singularity analysis~\cite{FlSe09}, aided by
  estimates from the related models restricted to the upper-half
  plane.  As $S(t)=\left(\frac{1}{1-|\mS|t}\right)(1-2t
  S_{0,1}(t))$, we show that $S(t)$ admits a simple pole at
  $|\mS|^{-1}$, and that this is the dominant singularity. The
  asymptotic expression in Equation~\eqref{eq:sym_asmpt} is then a
  consequence of evaluating the residue at this value, and bounding
  the dominant singularity of $S_{0,1}(t)$.

To accomplish this, we first consider the class of walks with steps
from $\mS$ which remain in the upper half plane, and return to the
$x$-axis. This is a well studied class, and the methods of Banderier
and Flajolet~\cite{BaFl02} yield the following expression for the generating function $H(t)$:
\[ H(t) =  \frac{ (1-tp_0) - \sqrt{\left(1-tp_0\right)^2 -4t^2p_{-1}p_{1}} } {2tp_1p_{-1}}. \]
In particular, the dominant singularity of $H(t)$ is
$t=\frac{1}{p_0+2\sqrt{p_1p_{-1}}}$. Now, the set of quarter plane
walks with steps from $\mS$ which return to the $x$-axis is a subset of
this set, and so
\[
[t^n]S_{0,1}(t) \leq [t^n] H(t).
\]
Consequently, $S_{0,1}(t)$ is convergent for $0\leq t \leq
{p_0+2\sqrt{p_1p_{-1}}}<|\mS|$, where the latter inequality is a result of the
fact that $p_{-1} = 1$, $p_0\in\{0,1\}$ and $p_1\in \{1,2,3\}$ in
these cases. Thus, the singularity at $|\mS|^{-1}$ is indeed dominant.

We also need to verify that $S_{0,1}(|\mS|^{-1})\neq 0$ to justify
that $|\mS|^{-1}$ is not a removable singularity.  In
Section~\ref{sec:explicit}, we determine an explicit expression for
$1/Y_n$.  Substituting $t=1/|S|$ into this expression proves that
$Y_n(1)Y_{n+1}(1)$ is monotonically decreasing, so that the error on
the $N^{\text{th}}$ partial sum of the alternating series is bounded
by $Y_{N+1}(1)Y_{N+2}(1)$.  Numerically evaluating the $10^\text{th}$
partial sum is sufficient to bound $1 - 2 \sum_n
(-1)^{n}Y_n(1)Y_{n+1}(1)$ away from 0 in each case. The results are
summarized in Table~\ref{tab:estimates-sym}.

We have shown that the dominant singularity of $S(t)$ is indeed the simple pole at
$|\mS|^{-1}$. The residue is $S_{0,1}(|\mS|^{-1})$ which, when
evaluated with suitable precision, gives the stated constant in
Table~\ref{tab:estimates-sym}. The sub-dominant factor comes from the
inverse of the dominant singularity of $S_{0,1}(t)$, which is bounded by $\frac{1}{p_0+2\sqrt{p_1p_{-1}}}$.

\end{proof}

\begin{table}
\greybox{
\begin{tabular}{lll}
{\bf Model} \hspace{0.2in} & {\bf Asymptotic estimate for number of walks
of length $n$} \\[2mm]
$\mA$ & $A_n\sim \kappa_A 3^n + O\left(\left( 2\sqrt{2} \right)^n\right)$ & $\kappa_A = 0.17317888\dots$\\
$\mB$ & $B_n\sim \kappa_B 4^n + O\left(\left(1+ 2\sqrt{2} \right)^n\right)$ & $\kappa_B = 0.15194581\dots$ \\
$\mC$ & $C_n\sim \kappa_C 5^n + O\left(\left(1+ 2\sqrt{3} \right)^n\right)$ &  $\kappa_C = 0.38220125\dots$
\end{tabular}}
\smallskip
\caption{Asymptotic estimates for number of walk of length
  $n$.  On a modern computer the $\kappa_S$ can be calculated to a thousand decimal places in seconds.}
\label{tab:estimates-sym}

\end{table}

\subsection{The generating functions $A(t)$, $B(t)$ and $C(t)$ are not D-finite }
\label{sec:not-dfinite-sym}
The set of D-finite functions are closed under algebraic
substitution. Thus, since our goal is to prove that the generating
functions $A(t)$, $B(t)$ and $C(t)$ are not D-finite, it is sufficient
to consider these functions evaluated at $t= q/(1+q^2)$. These turn
out to be easier to analyze as the transformation concentrates
the singularities around the unit circle. \emph{As such, we shall
  re-interpret the notation we have introduced thus far to be
  functions of $q$ directly.}

For each model, the $Y_n(1)$ terms contribute singularities. A quick
glance at an example is very suggestive; see Figure~\ref{fig:sings}
for the singularities of $Y_{20}(1)$ in the $q$-plane for the three
different models.
\begin{figure}
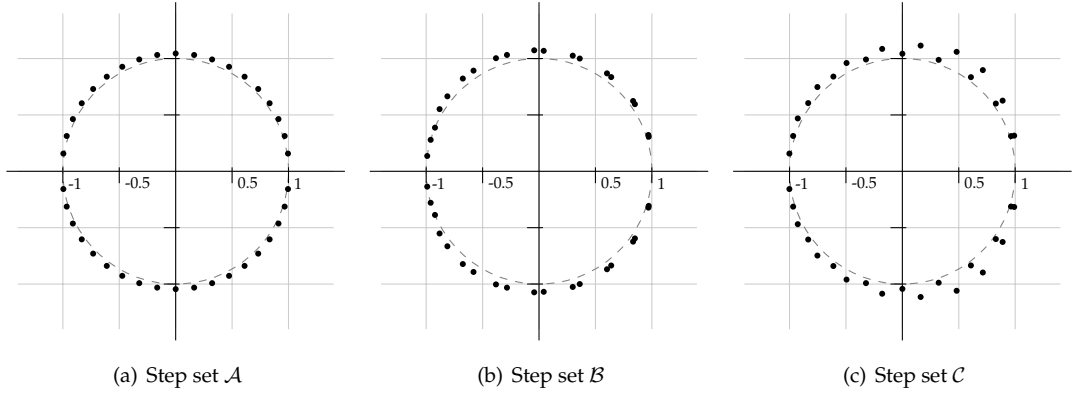
\center
\mbox{ \subfigure[Step set \mA]{\pointplot{\aYpos}} 
       \subfigure[Step set \mB]{\pointplot{\bYpos}} 
       \subfigure[Step set \mC]{\pointplot{\cYpos}} }
\caption{Plots of the singularities of
  $Y_{20}(1)|_{t=\frac{q}{1+q^2}}$ for the three symmetric models.}
\label{fig:sings}
\end{figure}
The main difficulty is proving that the singularities are
genuinely present in the generating function.  To prove this we follow
these steps:

\begin{description} \setlength{\itemsep}{0mm} \sl
\item[Step 1] Determine an explicit expression for $Y_n(1)$;
\item[Step 2] Determine a polynomial $\sigma_n(q)$ whose set of roots contains the poles of $Y_n(1)$;
\item[Step 3] Determine a region where there are roots of $\sigma_n(q)$ that are truly poles of $Y_n(1)$; 
\item[Step 4] Show that there is no point $\rho$ in that region that is a root of both $\sigma_n(q)$ and $\sigma_k(q)$ for different $n$ and $k$;
\item[Step 5] Demonstrate that~$S_{1,0}(q/(1+q^2))$ has an infinite number of  singularities and, consequently, is not D-finite. It follows that $S(t)$ is not D-finite, by closure under algebraic substitution and the expression in Equation~\eqref{eq:S11-sym}. 
\end{description}

\subsubsection{Step 1: An explicit expression for $Y_n$}
\label{sec:explicit}

In this section we find an explicit, non-iterated expression for the functions~$Y_n$. We follow the method of~\cite{MiRe09} very closely, with the exception that we make the variable substitution earlier in the process. As such, we repeat, we view all functions as functions of $q$ in this section. 

We begin by performing the variable substitution $t=q/(1+q^2)$ directly into Equation~\eqref{eqn:fund}, and re-solve the kernel to ensure control over the choice of the branch in the solution. The kernels are:
 \begin{align*}
\text{Model } \mA\qquad   K(x,y) &= - q(x^2+1)\,y^2 + x(1+q^2)\,y - qx^2 \\
\text{Model } \mB\qquad    K(x,y) &= - q(x+1)\,y^2 + x(-qx+1+q^2)\,y - qx^2 \\  
\text{Model } \mC\qquad   K(x,y) &= - q(1+x+x^2)\,y^2+x(-qx+1+q^2)\,y-qx^2.
  \end{align*}
Recall we denote this generically as $K(x,y)=a_2y^2 +a_1y  +a_0$, adapting the $a_i$ to each particular model. 
Each is solved as before to get our initial solutions to $K(x, Y(x))=0$. Great care is taken here to ensure that the branch as written remains analytic at 0:
\begin{align*}
\text{Model } \mA\qquad {Y}_{\pm 1}(x;q) &= \frac{x}{2q(1+x^2)} \cdot \left( 1+q^2\mp\sqrt{1 - 2(2x^2+1)q^2 + q^4} \right)\\
\text{Model } \mB\qquad{Y}_{\pm 1}(x;q) &=  \frac{x}{2q(1+x)} \cdot \left( 1-qx+q^2 \mp\sqrt{q^4-2q^3x+(x^2-4x-2)q^2-2qx+1} \right)\\
\text{Model } \mC\qquad{Y}_{\pm 1}(x;q) &= \frac{x}{2q(1+x+x^2)} \cdot \left( 1-qx+q^2\mp\sqrt{ q^4-2q^3x-(3x^2+4x+2)q^2-2qx+1} \right).
  \end{align*}
We define the sequence of iterates $\{Y_n(x)\}_{(n)}$ as before: $Y_{n+1}(x)=Y_+(Y_n(x);q),\quad Y_1(x)=Y_+(x;q)$. 

For each of these models, examining the coefficients of $y$ in the kernel implies 
\begin{equation} \label{eq:rec-base} \frac{1}{Y_-(x;q)} +
  \frac{1}{Y_+(x;q)} = \frac{Y_-(x;q) + Y_+(x;q)}{Y_{-}(x;q) \cdot
    Y_+(x;q)} =
  \frac{-a_1/a_2}{a_0/a_2}=-\frac{a_1}{a_2}. \end{equation} 
  
Furthermore, the iterates compose nicely due to the following lemma.
\begin{lemma} For each of the symmetric models $\mA,\mB,\mC$ we have
   \[ Y_-\left(Y_+(x)\right) = Y_+\left(Y_-(x)\right) = x.\]
\end{lemma}
\begin{proof}
For any given model, expanding the polynomial
\[p(z) = \prod_{(j,k)\in\{\pm1\}^2} (z - Y_j\left(Y_k(z)\right)),\]
implies $p(z) = (x-z)^2r(z)$, where $r(z) \in
\mathbb{R}(x,t)[z]$ is given by
{\small 
\begin{align*}
 \quad& r(z) \\
\mA \quad& \left((t^2+x^2)z^2-x(1-2t^2)z +t^2x^2\right)\left(t^2+x^2\right)^{-1} \\
\mB \quad& \left((t^2x^2+2t^2x+tx+tx^2+t^2)z^2+(-x+2t^2x^2+tx^2+2t^2x)z+t^2x^2\right)\left(t(x+1)(tx+x+t)\right)^{-1} \\
\mC \quad& {\left((2t^2x+t^2+tx^2+tx+x^2+t^2x^2)z^2+(2t^2x^2+tx^2+2t^2x-x)z+t^2x^2\right)\left(2t^2x+t^2+tx^2+tx+x^2+t^2x^2\right)^{-1}}\\
\end{align*}  
}
Thus, two of $p(z)$'s roots are equal to $x$ and examination of the initial terms of a Taylor series in~$t$ shows that 
$Y_+(Y_+(x))$ and $Y_-(Y_-(x))$ are not. 
\end{proof}

It turns out to be easier to work with the reciprocal of $Y_n$, so we define
$\y{n}= \frac{1}{Y_n(1)}$
and view this as a function of $q$. Equation~\eqref{eq:rec-base} then
converts into a recurrence after the substitution~ $x =
Y_{n-1}(x)$. Specifically, this gives a linear recurrence for the reciprocal
function, $\frac{1}{Y_n(x)}$; we are interested in this evaluated at
$x=1$, and the resulting recurrences and their solutions in terms of
$\y{1}$ are summarized in Table \ref{tab:sym-rec}.

\begin{table}
\begin{tabular}{cll}
Model& {\bf Recurrence} &$\yq{n}{q}$\\[2mm]
\mA & $\y{n} = (q+\frac{1}{q})\y{n-1}-\y{n-2}$ &$\frac{(q^2-q^{2n}) + q(q^{2n}-1)\overline{Y}_{1}}{q^n(q^2-1)}$
\\[2mm]
\mB, \mC & $\y{n} = (q+\frac{1}{q})\y{n-1}-\y{n-2}-1\quad$ &$\frac{q(q-1)(q^{2n}-1)\overline{Y}^{B,C}_{1} + (q-q^n)(2q^{n+1}-q^n+q^2-2q)}{q^n(q+1)(q-1)^2}$\\[2mm]
\end{tabular} 
\caption{The recurrences and solutions for models \mA, \mB, and \mC.}
\label{tab:sym-rec}
\end{table}

Following the same procedure as above, we obtain a generic expression
for $S_{1,0}(t)$, the generating function for the number of walks
which return to the axis for model $\mS$, which can be applied to all
three symmetric walks:
\begin{equation}\label{eq:S11-sym-q} S_{1,0}\left(\frac{q}{1+q^2}\right) =
  (q+1/q)\Sum_{n=0}^\infty (-1)^n Y_{n}(1)
  {Y}_{n+1}(1).\end{equation} Our careful choice of
branches now implies that this is a formal power series. (Remark, this
was not the case in~\cite{MiRe09}.)  Our expression is robust -- the
sum converges everywhere, except possibly on the unit circle and at
the poles of the $Y_n$.
\begin{prop}
\label{thm:converge}
For each model in $\{\mA, \mB, \mC\}$ the sum $(q+1/q)\Sum_{n=0}^\infty (-1)^n Y_{n}(1) {Y}_{n+1}(1)$ is convergent for all~$q\in \mathbb{C}$ with $|q| \neq 1$, except possibly at the set of points defined by the singularities of the  $Y_{n}(1)$ for all $n$.
\end{prop}
\begin{proof}
In all cases the ratio test is applied to the explicit formulas for $\overline{Y}_n(1;q)$. It is a mechanical exercise to verify that for each case
when~$|q|<1$:
      \[ \lim_{n\rightarrow\infty}\left| \frac{Y_{n+1}Y_{n+2}}{Y_nY_{n+1}} \right|  =\left| \frac{\overline{Y}_n}{\overline{Y}_{n+2}}\right|= |q|^2 < 1,\]
and when~$|q|>1$: 
      \[ \lim_{n\rightarrow\infty }\left| \frac{Y_{n+1}Y_{n+2}}{Y_nY_{n+1}} \right| = \left| \frac{\overline{Y}_n}{\overline{Y}_{n+2}}\right|  = \frac{1}{|q|^2} < 1.\]
\end{proof}

\subsubsection{Step 2: The singularities of $Y_n(1)$}
\label{sec:sings}

In order to argue about the singularities, we find a family of
polynomials $\sigma_n(q)$ that the roots of $\y{n}$ satisfy: the polynomials in Table~\ref{tab:polys} are obtained by manipulating
the explicit expressions given above. Unfortunately, extraneous roots
are introduced during the algebraic manipulation when an equation is
squared to remove the square root present. In fact, the extraneous
roots are exactly those which correspond to a negative sign in front
of the square root.  If one defines $\y{-n} = \overline{Y}_{-1} \circ
\y{-(n-1)}$ for $n>1$, then using the same argument
as above one can check that $\y{-n}$ satisfies the same recurrence
relation as $\y{n}$, up to a reversal of the sign in front of the
square root. Thus, we see that the set of roots of $\sigma_n(q)$ is
simply the union of the sets of roots of $\y{n}$ and $\y{-n}$.

\begin{table}\center
\begin{tabular}{ll}
Model & {$\sigma_n(q)$} \\[2mm]
\mA & {\small$\alpha_n(q)=q^{4n}+q^{2n+2}-4q^{2n}+q^{2n-2} + 1$} \\
\mB &{\small$\beta_n(q)=\left(q^{2n-1}+(q^3-2q^2-2q+1)q^{n-2}+1\right) \left(q^{2n+1}+(q^3-2q^2-2q+1)q^{n-1}+1\right)$ }\\
\mC &{\small $\gamma_n(q)=q^2(1+q^2-q)(1+q^{4n})+q(q^2-3q+1)(q+1)^2(q^n+q^{3n})$}\\
       &\hfill{\small $\qquad+q^{2n}(1-q^2-4q+14q^3-4q^5-q^4+q^6)$}\\
\end{tabular}
\caption{The singularities of $Y_n$ in the $q$-plane satisfy the
  polynomial equation~$\sigma_n(q)=0$.}
\label{tab:polys}
\smallskip
\end{table}

Furthermore, we can show that these roots are dense around the unit circle using the results of Beraha, Kahane, and Weiss -- specifically, a  weakened statement of the main theorem of~\cite{BeKaWe78}.

\begin{prop}[\cite{BeKaWe78}] Given non-zero polynomials $\mu_1,\dots,\mu_k,\lambda_1,\dots,\lambda_k$, define
\[P_n(q) = \sum_{j=1}^k \mu_j(q)\lambda_j(q)^n.\]
If there does not exist a constant $\omega$ such that $|\omega|=1$ and $\lambda_j=\omega\lambda_k$ for $j\neq k$, and for some $l\geq 2$
\[ \left| \lambda_1(x) \right| = \left| \lambda_2(x) \right| = \cdots = \left| \lambda_l(x) \right| > \left| \lambda_j(x) \right|, \]
for all  $l+1 \leq j \leq k$, then $x$ is a limit point of the zeroes of $\{ P_n(q)\}$ -- i.e., there exists a sequence $q_n$ converging to $x$ such that $P_n\left(q_n\right)=0$ for all $n$.
\end{prop}

As each of $\alpha_n(q),\beta_n(q),$ and $\gamma_n(q)$ can be decomposed into the required form where the $\lambda_j(q)$ are simply powers of $q$, and thus have the same modulus when $q$ is on the unit circle, this immediately gives the following result.

\begin{cor} The roots of the families of polynomials $\{\alpha_n(q)\},\{\beta_n(q)\},$ and $\{\gamma_n(q)\}$ are dense around the unit circle. \label{cor:dense} \end{cor}

Furthermore, as our results on the convergence of the series in Equation \eqref{eq:S11-sym-q} is only valid at points off of the unit circle, we use a Lemma of Konvalina and Matache to determine when the roots of $\alpha_n,\beta_n,$ and $\gamma_n$ may lie on the unit circle.

\begin{lemma}[Konvalina and Matache \cite{KoMa04}, Lemma
  1] \label{lem:ucirc} Suppose $F(x)$ is a palindromic polynomial
  (its coefficient sequence is the same when read from the left or
  right) of degree $2N$.  Then the argument of any root of $F(x)$
  which lies on the unit circle satisfies
\[ \phi(\theta) = \epsilon_{N} + 2\sum_{k=0}^{N-1}\epsilon_k \cos\left( (N-k)\theta \right)  \]
where $\epsilon_j$ denotes the coefficient of $x^j$ in $F(x)$.
\end{lemma}

Applied to our polynomials, it gives the following.

\begin{prop} \label{lem:uargB} 
 For all natural numbers $n$, $\alpha_n(q)$ and $\gamma_n(q)$ have no roots on the unit circle, except possibly $q=\pm1$.  Furthermore, if $q$ is a root of $\beta_n(q)$ on the
  unit circle not equal to 1 then
  \[ \arg{q} \in \left[\pi-\arccos\left(\sqrt{2}-\frac{1}{2}\right),\pi\right) \bigcup \left[-\pi,-\pi+\arccos\left(\sqrt{2}-\frac{1}{2}\right)\right). \]
\end{prop}
\begin{proof}
  As $\alpha_n(q),\beta_n(q)$, and $\gamma_n(q)$ are palindromic, Lemma~\ref{lem:ucirc} implies, after some trigonometric simplification, that
  the argument of any root $q$ on the unit circle satisfies
\begin{align*}
  \phi_A(\theta) &= \mathbf{X} + 2\cos^2(\theta)-3 \\
  \phi_B(\theta) &= {\mathbf{X}}^2+\left(2\cos^2(\theta)-\cos(\theta)-3\right){\mathbf{X}}+2\cos^3(\theta)-4\cos^2(\theta)-\cos(\theta)+4  \\
  \phi_C(\theta) &= 2\left(2\cos(\theta)-1\right){\mathbf{X}}^2+\left(4\cos^2(\theta)-2\cos(\theta)-6\right){\mathbf{X}}+4\cos^3(\theta)-8\cos^2(\theta)-6\cos(\theta)+12 
\end{align*}
  respectively, where ${\mathbf{X}}=\cos(n\theta)$.
  
   It is easy to see that $\phi_A(\theta)=0$ only if $\theta=0$ or $\theta=\pi$.  For the other models, in order to give a bound on where the roots of each expression lie we treat~${\mathbf{X}}$ as an independent real variable lying in the range $[-1,1]$ for some fixed value of~$\theta$, and determine where the minimum value of the above expression is at most zero.  

First consider $\phi_B(\theta)$.  As this is a quadratic in
${\mathbf{X}}$ it attains its minimum value either at
${\mathbf{X}}=\pm1$ or when
${\mathbf{X}}=-(2\cos^2(\theta)-\cos(\theta)-3)/2$.

Substituting~${\mathbf{X}}=\pm1$ into $\phi_B(\theta)$ yields expressions which are always
greater than zero when $\theta \notin \{0,\pi\}$.  Furthermore, when ${\mathbf{X}}=-(2\cos(\theta)^2-\cos(\theta)-3)/2$ our expression for
$\phi_B(\theta)$ simplifies to
\[ \frac14 \left( -4\cos^4(\theta)+12\cos^3(\theta)-5\cos^2(\theta)-10\cos(\theta)+7\right). \]
One can verify that this is at most zero only when $\theta=0$ or
\[\theta \in \left[\pi-\arccos\left(\sqrt{2}-\frac{1}{2}\right),\pi\right) \bigcup \left[-\pi,-\pi+\arccos\left(\sqrt{2}-\frac{1}{2}\right)\right). \]

As $\phi_C(\theta)$ is also a quadratic in $\mathbf{X}$ an analogous argument shows that $\phi_C(\theta)=0$ only if $\theta \in \{0,\pi\}$, as desired.
\end{proof}

Thus, every point on the unit circle is a limit point of each of the sets
\[ \{q : \alpha_n(q)=0 \text{ for some $n$}\}, \{q : \beta_n(q)=0 \text{ for some $n$}\}, \text{ and } \{q : \gamma_n(q)=0 \text{ for some $n$}\},\]
but no element of these sets lies on the unit circle (except in a special region when dealing with model $\mathcal{B}$).  In fact, as the polynomials are palindromic, a straightforward application of Rouch\'e's theorem proves that all roots of $\alpha_n,\beta_n,$ and $\gamma_n$ converge to the unit circle as $n$ approaches infinity.

\subsubsection{Step 3: Verify that $Y_n(1)$ has some singularities}
\label{sec:roots}
At this point we have not yet completely established that the $Y_n(1)$
actually have singularities. Theoretically, it is possible that all
the roots were added in our manipulations to determine $\sigma_n(q)$
for the different models (as mentioned above, the roots of
$\sigma_n(q)$ are either singularities of $Y_n(1)$ or singularities of
$Y_{-n}(1)$). Thus, we prove Lemma~\ref{lemma:rootABC} which describes
at least some region where we are certain to find roots of
$\y{n}$. Experimentally, it seems that the roots are evenly
partitioned so that those outside the unit circle belong to~$\y{n}$
and those inside the unit circle belong to~$\y{-n}$, but we do not
prove this.

\begin{lemma} \label{lemma:rootABC} For each model, if $\arg(q) \in (-\pi/2,-3\pi/8)
  \cup (3\pi/8,\pi/2)$ then $\y{n} = \y{-n}|_{q\mapsto 1/q}$ for
  all $n$. Consequently $Y_n$ admits at least one singularity in the complex
  $q$-plane in that region, for an infinite number
  of $n$.
\end{lemma}
The proof requires only basic manipulations of the formulas. 
We offer the proof for the $\mA$ case, the other two are similar. 

\begin{proof}[$\mA$ case] First, note that it is sufficient to prove the result for $\arg(q) \in (-\pi/2,-3\pi/8)$.  We claim that in this region the identity $\y{1}=\y{-1}|_{q\mapsto 1/q}$ holds, which is equivalent to proving $q^2\sqrt{ (q^4-6q^2+1)/q^4 } = -\sqrt{q^4-6q^2+1}$ (using the standard branch-cut of the square root). 
  
    If $q=re^{i\theta}$ in polar form it is straight forward to verify $\Re\left( q^4-6q+1 \right),\Im\left( q^4-6q+1 \right) \geq 0$ for the values of $\theta$ under consideration, so $\arg(q^4-6q^2+1) \in [0,\pi/2]$.  Furthermore, for these values of $\theta$ we have $\arg(1/q^4) \in (-\pi/2,0)$ so $\arg(q^4-6q^2+1) + \arg(1/q^4) \in (-\pi/2,\pi/2]$, and
    \[ q^2\sqrt{ (q^4-6q^2+1)/q^4 } = q^2\sqrt{1/q^4}\sqrt{(q^4-6q^2+1)}.\]
    By our choice of region, $\arg(1/q^4) = -4\theta-2\pi$ and thus $q^2\sqrt{1/q^4}=r^2e^{2i\theta}\cdot\frac{1}{r^2}e^{i(-4\theta-2\pi)/2} = e^{-\pi i} =-1$, proving the result on $\y{1}$.

Given this, we note
\[ \y{-n} = \frac{(q^2-q^{2n}) + q(q^{2n}-1)\y{-1}}{q^n(q^2-1)}=\frac{q^{1-n}-q^{n-1}}{q-q^{-1}}+\frac{q^n-q^{-n}}{q-q^{-1}}\y{-1},
\]
so that
\[ \left.\y{-n}\right|_{q\rightarrow1/q} =\frac{q^{1-n}-q^{n-1}}{q-q^{-1}}+\frac{q^n-q^{-n}}{q-q^{-1}}\y{1},\]
by the base case, as the rest is invariant, and thus $\y{-n}(1/q) = \y{n}(q)$.
\end{proof}

As the region considered above is disjoint from the region where the roots of $\beta_n(q)$ lie on the unit circle, all the singularities we have found lie off of the unit circle.

%

\subsubsection{Step 4: The singularities are distinct} 
\label{sec:Distinct}
We prove that the poles are distinct when they lie off of the unit
circle by determining expressions for the powers of $q$ at the poles
of the $Y_n$.

\begin{prop} \label{prop:Distinct} For models $\mathcal{A}$ and
  $\mathcal{C}$, if $q_n$ is a pole of $Y_n$ which lies off of the
  unit circle then it is not a pole of $Y_k$ for $k\neq n$. For model
  $\mathcal{B}$, if $q_n$ is a pole of $Y_n$ off of the unit circle
  then it is not a pole of $Y_k$ for $|k-n|>1$.
\end{prop}
\begin{proof}
  For each of the three models we find the roots of the numerators of
  our explicit expressions in Table~\ref{tab:sym-rec} as quadratics in
  $q^n$.  This determines functions $r_1(q)$ and $r_2(q)$, independent
  of $n$, such that $q_n^n=r_1(q_n)$ or $q_n^n=r_2(q_n)$ at any pole
  $q_n$ of $Y_n$.  

  Now, suppose $q_n$ is also a pole of $Y_k$
  for $k \neq n$, so that $q_n^k=r_1(q_n)$ or $q_n^k=r_2(q_n)$.  If 
  $q_n^k = r_1(q_n) = q_n^n$ or $q_n^k = r_2(q_n) =
  q_n^n$ then it is immediate that $q_n$ must be on the unit circle.  Thus we may assume,
  without loss of generality, that $q_n^n = r_1(q_n)$ and
  $q_n^k=r_2(q_n)$ -- we consider each model separately.

\begin{itemize}
\item \textit{(Model \mA)} Here, 
\[ r_1(q),r_2(q) = \pm \frac{2 \sqrt{-q^2}}{q^2+\sqrt{1-6q^2+q^4}-1}, \]
so that $q_n^{n-k} = r_1(q_n)/r_2(q_n) = -1$, implying that $q_n$ must lie on the unit circle.
\item \textit{(Model \mB)} Here, 
\[ r_1(q) = \frac{2q^2}{2q^2+\sqrt{q^4-2q^3-5q^2-2q+1}-2q+q\sqrt{q^4-2q^3-5q^2-2q+1}-1-q^3},\]
and $r_2(q) = r_1(q)/q$, so that $q_n^{n-k} = r_1(q_n)/r_2(q_n) = q_n$.  Thus, either $n = k + 1$ or $q_n$ lies on the unit circle.

\item \textit{(Model \mC)} In this slightly trickier case we have
\begin{align*} 
r_1(q) &= \frac{q(-1-q-i\sqrt{3}+\sqrt{3}q)} {-2q^2-\sqrt{1-2q-9q^2-2q^3+q^4}-2q+q\sqrt{1-2q-9q^2-2q^3+q^4}+1+q^3} \\
r_2(q) &= \frac{q(-1-q+i\sqrt{3}-\sqrt{3}q)} {-2q^2-\sqrt{1-2q-9q^2-2q^3+q^4}-2q+q\sqrt{1-2q-9q^2-2q^3+q^4}+1+q^3},
\end{align*}
which implies
\[ q_n^{n-k} =  r_1(q_n)/r_2(q_n) = e^{-2\pi i/3} + \sqrt{3}\frac{e^{-\pi i/6}}{q-e^{\pi i/3}}.\]
Substituting $q_n=re^{i\theta}$ into this expression allows one to see that the right hand side has modulus greater than or equal to one when $\theta \in [0,\pi)$ and modulus less than or equal to one when $\theta \in (-\pi,0]$.  

Suppose now that $n > k$.  If $|q_n| < 1$ then $|q_n|^{n-k} < 1$ and $q_n$ cannot lie above the real axis (as the modulus of the right hand side would be greater than or equal to 1).  Similarly, if $|q_n| > 1$ then $|q_n|^{n-k} > 1$ and $q_n$ cannot lie beneath the real axis.  As we take the principal branch of the square root in the definition of $\y{n}(q)$, we have $\y{n}(q)^{*} = \y{n}(q^*)$ (where $q^*$ denotes the complex conjugate of $q$) so there are, in fact, no solutions off of the unit circle or real axis.  One can easily verify that there are no non-unit real solutions, and when $n<k$ the argument is analogous.
\end{itemize}
\end{proof}

\subsubsection{Step 5: The generating function is not D-finite}  Now we tie up all the arguments.

\begin{theorem}\label{thm:sym-nondfinite}
  The generating functions $A(t)$, $B(t)$, $C(t)$ of walks in the quarter
  plane with steps from $\mA$, $\mB$, and $\mC$ respectively, are all
  not D-finite.
\end{theorem}

\begin{proof} 
We now show that the infinite set of poles described in
Proposition~\ref{prop:Distinct} are indeed poles of the generating
function. The D-finiteness result follows from this and the fact that D-finite
functions have only a finite number of singularities.

Fix a model. For each $n$, there is at least one choice of $q_n$
amongst the poles of $Y_n(1)$ which is not a pole of any $Y_k(1)$ for
$|n-k|>1$. This is a direct consequence of
Proposition~\ref{prop:Distinct}. 

Next, we break the main sum of Equation~\eqref{eq:S11-sym-q} into
three parts, and examine the behaviour at $q_n$ -- let us first consider the cases of models $\mA$ and $\mC$. The sum
is decomposed as follows:
\[ \frac{q}{1+q^2}\cdot S_{0,1} = \underbrace{\sum_{k=0}^{n-2}(-1)^k
  Y_kY_{k+1}}_{\text{a finite sum}} +
\underbrace{(-1)^{n-1}Y_n(Y_{n-1}-Y_{n+1})}_{\text{pole contribution}}+
\underbrace{\sum_{k\geq n+1}(-1)^k Y_kY_{k+1}}_{\text{convergent at
  }q=q_n}.\]

The initial
and terminal sums do not admit poles at $q_n$ since
Proposition~\ref{prop:Distinct} implies in these two cases  $Y_k$ does not have a pole at
$q_n$ for $k \neq n$, and an argument identical to the proof of
Proposition~\ref{thm:converge} implies the second summation is convergent at
this point.  Furthermore, if we substitute $q_n$ into the
corresponding recurrence from Table~\ref{tab:sym-rec}, and recall it
is a zero of $\overline{Y}_n$,
we derive $\overline{Y}_{n+1}(q_n) = - \overline{Y}_{n-1}(q_n) + \epsilon$ (where $\epsilon=0$ for model \mA~and $\epsilon=1$ for models \mB~and \mC), and so $Y_{n-1}-Y_{n+1}\neq0$. We can then conclude that
$q_n$ is a pole of the series, for all $n\geq 1$. 

Thus, we have shown that both  $A\left(\frac{q}{1+q^2}\right)$ and $C\left(\frac{q}{1+q^2}\right)$
have an infinite number of poles, and are not D-finite. The stated
result follows immediately from the fact that the class of D-finite functions is closed under algebraic substitution.

The remaining case of model~\mB~is almost identical, save for the fact
that $Y_n$ and $Y_{n-1}$ share some, but not all, of their poles.  If $q_n$ is not a pole
of $Y_{n+1}$, then the argument above shows that it is a pole of
$B\left(\frac{q}{1+q^2}\right)$.  If $q_n$ is a pole of $Y_{n+1}$,
then the summand $(-1)^nY_nY_{n+1}$ has a pole of a larger order than
the other two summands in which that pole appears, and it
cannot be cancelled by the rest of the summation. This means that 
$q_n$ is again a pole of $B\left(\frac{q}{1+q^2}\right)$, and the
remainder of the argument is as for the other two cases. 
\end{proof}

\subsection{Return to the $t$-plane}
It is useful to visualize the singularities in the $t$-plane as well,
as they control the sub-dominant asymptotics. Figure~\ref{fig:singsT}
contains precisely such a plot. 

The sub-dominant singularities in the $t$-plane appear to converge to
$t=1/2$; in fact for model $\mC$ there are two singularities of
$Y_2(t)$ which have modulus exactly 1/2: $t = -1/4\pm\sqrt{3}i/4$.

\begin{center}\begin{figure}
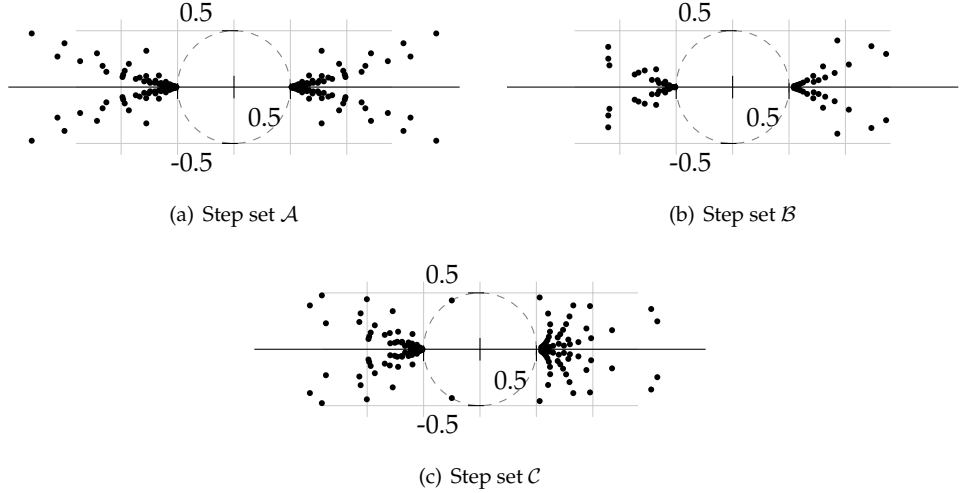

\mbox{ \subfigure[Step set \mA]{\pointplotT{\aTpts}} \quad 
       \subfigure[Step set \mB]{\pointplotT{\bTpts}}}
\mbox{ \subfigure[Step set \mC]{\pointplotT{\cTpts}} }
\caption{All singularities from $Y_1(t),\dots,Y_{15}(t)$ for the symmetric models; these form the sub-dominant singularities of $S_{1,1}(t)$.  The curve $|t|=1/2$ is sketched.}
\label{fig:singsT}
\end{figure}\end{center}

\section{Asymmetric models} 
\label{sec:asymmodels}
The asymmetric models are not substantially different, but when we
iterate we have more functions to track. 
Aside from some irritating bookkeeping, there is no main
obstacle to following the strategy of the symmetric models.
  
\subsection{An explicit generating function expression}
To obtain the generating function expressions we follow the same path
as in the symmetric case: we generate a sequence of equations, each
which annihilates the kernel. This opens up the possibility of a
telescoping sum expression from which we can find an expression for
the generating function of walks returning to the axis in terms of iterates. 
An explicit expression for these iterates is obtained
by solving some very simple recurrences. We complete these steps for
the asymmetric models in this section.

As before, we begin with the main functional
equation~\eqref{eqn:fund}, make the substitution~$t=q/(1+q^2)$, and
re-arrange to get the kernel equations: 
\begin{align*}
    \left(xy(1+q^2)-qy^2-qxy^2-qx^2\right)D_{x,y}(t) &= xy(1+q^2) - qx^2D_{x,0}(t)- qy^2D_{0,y}(t) \\
    \left(xy(1+q^2)-qy^2-qxy^2-qx^2y^2-qx^2\right) E_{x,y}(t) &= xy(1+q^2) - qx^2E_{x, 0}(t) - qy^2E_{0, y}(t),
\end{align*}
with kernels
\begin{align*}
   \text{Model } \mD:  \qquad K(x,y) &= -q(1+x)y^2+(1+q^2)xy-qx^2 \\
      \text{Model } \mE:  \qquad K(x,y) &= -q(1+x+x^2)y^2+(1+q^2)xy-qx^2.
\end{align*}
As there is no longer an $x=y$ symmetry, we solve the
kernels as functions of both $x$ and $y$; that is, we find $Y(x) $
satisfying $K(x,Y(x))=0$ and also $X(y)$ satisfying $K(X(y), y)=0$.
We have some choice over how we split the solutions over different
branches.  One such choice of branches is: 
{ \begin{align*}
\text{Model } \mD:\quad\\
     {X}_{\pm}(y;q) &= \frac{y}{2q} \cdot \left( 1-qy+q^2\mp\sqrt{q^4-2q^3y+(y^2-2)q^2-2qy+1} \right)\\
     {Y}_{\pm}(x;q) &= \frac{x}{2q(1+x)}\cdot \left( 1+q^2\mp\sqrt{
         q^4-4q^2x-2q^2+1} \right), \\
\text{Model } \mE:\quad\\
     {X}_{\pm}(y;q) &=  \frac{y}{2q(1+y^2)} \cdot \left( 1-qy+q^2 \mp\sqrt{ q^4-2q^3y-(3y^2+2)q^2-2qy+1 } \right)\\
     {Y}_{\pm}(x;q) &= \frac{x}{2q(1+x+x^2)} \cdot \left(
       1+q^2\mp\sqrt{ q^4-2(2x^2+2x+1)q^2+1} \right).
  \end{align*}}
Next, as we described in the introductory summary, we repeatedly alternate the substitution of the $X$ and $Y$ and create two related sequences of functions:
\begin{equation*}
 \chi_n(x) = X_+ (Y_+ (\chi_{n-1}(x);q);q), \quad \chi_0(x)=x \quad\text{and}\quad
 \Upsilon_n(y) = Y_+ (X_+(\Upsilon_{n-1}(y);q);q),  \quad \Upsilon_0(y)=y.
\end{equation*}
Simple substitutions yield the kernel relations
\[K(\chi_n(x),
Y_+(\chi_n(x)))=K(X_+(\Upsilon(y)),\Upsilon(y))=0,\] amongst others. As
before, we generate an infinite list of relations by substituting
$x=\chi_n(x)$, $y= Y_+(\chi_n(x))$, and then a second infinite list
using the substitutions~$x=X_+(\Upsilon(y))$, $y=\Upsilon(y)$. Again,
we form a telescoping sum, and after some manipulation this results in
an expression for the generating functions of the walks returning to the axis. For
$S\in\{D, E\}$ we have:
\begin{align}
S_{x,0}\left(\frac{q}{1+q^2}\right) &=
\frac{q}{1+q^2}\Sum_{n\geq 0}\chi_n(x)\cdot
\underbrace{\left(Y_+\circ\chi_n(x)-Y_+\circ\chi_{n-1}(x)\right)}_{\Delta_{L,n}(x)} \label{eq:Sx0}\\
S_{0,y}\left(\frac{q}{1+q^2}\right) &= 
\frac{q}{1+q^2}\Sum_{n\geq 0}X_+\circ\Upsilon_n(y)\cdot
\underbrace{\left(\Upsilon_n(y)-\Upsilon_{n+1}(y)\right)}_{\Delta_{R,n}(y)}.
\end{align}
The two models have identical structure in their generating function, and differ only in their respective functions $X_+$ and $Y_+$. 
Our greatest challenge at this point is keeping track of the various parts:
\begin{align*}
 \Delta_{L,n}(x) &= Y_+\circ\chi_n(x)-Y_+\circ\chi_{n-1}(x)&&\Delta_{R,n}(y) = \Upsilon_n(y) - \Upsilon_{n+1}(y)&&\\
 \Delta_{L,0}(x) &= Y_+(x) &&\Delta_{R,0}(y) = \Upsilon_0(y).
  \end{align*}
For each model we isolate the left and right hand sides, defining $L(x,q) = qx^2S_{x,0}(q/(1+q^2))$ and $R(y,q) = qy^2S_{0,y}(q/(1+q^2))$, so that
\[ S_{x,y}(q/(1+q^2)) = \frac{xy(1+q^2) - L(x,q) -
  R(y,q)}{K(x,y)}, \]
and the counting generating function has the form
\begin{equation}\label{eq:asymgf-q}
S(q/(1+q^2))= \frac{(1+q^2) - L(1,q) -R(1,q)}{1-K(1,1)}.
\end{equation}
For both asymmetric models we find an infinite set of points at which $L(1,q)$ is
singular, but~$R(1,q)$ is convergent.

Similar to previous cases, we can use the coefficients of $K(x,y)$
and the facts that
\[ Y_{\pm}\left(X_{\mp}(y)\right) = y \qquad X_{\pm}\left(Y_{\mp}(x)\right) = x, \] 
to form paired up recurrences
for the multiplicative inverses of these functions. Here we again use the
notation that $\overline{F}=\frac{1}{F}$:
\begin{equation}\label{eq:initialsystem}
\begin{aligned}
    \overline{\chi}_{n} &= (q+1/q)\overline{Y_+\circ\chi}_{n-1} -
    \overline{\chi}_{n-1}-1,&& 
 \qquad \overline{Y_+\circ\chi}_{n} = (q+1/q)\overline{\chi}_{n} - \overline{Y_+\circ\chi}_{n-1} \\
    \overline{\Upsilon}_{n} &=
    (q+1/q)\overline{X_+\circ\Upsilon}_{n-1} -
    \overline{\Upsilon}_{n-1},&& \qquad
    \overline{X_+\circ\Upsilon}_{n} = (q+1/q)\overline{\Upsilon}_{
      n}-\overline{X_+\circ\Upsilon}_{n-1}-1.
 \end{aligned}
\end{equation}

  Solving these recurrences, we obtain the closed form expressions: 
\begin{equation}\label{eq:closedform}
\begin{aligned}  
    \overline{\chi}_n &= \frac{ (q^{4n+3}-q^{4n+1}-q^3+q)\overline{Y}_+-2q^{4n+2}+q^{4n}+2q^{2n+2}+q^4-2q^2 }{ q^{2n}(q^2-1)^2 } \\[2mm]
    \overline{Y_+\circ\chi}_n &= \frac{ (q^{4n+4}-q^{4n+2}-q^2+1)\overline{Y}_+-2q^{4n+3}+q^{4n+1}+q^{2n+3}+q^{2n+1}+q^3-2q }{ q^{2n}(q^2-1)^2 } \\[2mm]
    \overline{\Upsilon}_n &= \frac{ (q^{4n+3}-q^{4n+1}-q^3+q)\overline{X}_+-q^{4n+2}-q^{4n+1}+q^{4n}+q^{2n+3}+q^{2n+1}+q^4-q^3-q^2 }{ q^{2n}(q^2-1)^2 } \\[2mm]
    \overline{X_+\circ\Upsilon}_n &= \frac{ (q^{4n+4}-q^{4n+2}-q^2+1)\overline{X}_+-q^{4n+3}-q^{4n+2}+q^{4n+1}+2q^{2n+2}+q^3-q^2-q }{ q^{2n}(q^2-1)^2 }.
 \end{aligned}
\end{equation} 

We next show that our expressions for $L(1,q)$ and $R(1,q)$ in terms of $\chi_n,  Y_+\circ\chi_n, \Upsilon_n,$ and $X_+\circ\Upsilon_n$ are valid for almost all of the complex plane.
  \begin{prop} \label{prop:AsymConv}
   For either
   $\mD$ or $\mE$, let $q \in \mathbb{C}$ such that $|q|\neq1$ and $\chi_n,\Delta_{L,n},\Delta_{R,n},$
   and $X_+\circ\Upsilon_n$ are all analytic. Then the related series $L(1,q)$ and $R(1,q)$ both converge for~$q\in\mathbb{C}$.
  \end{prop}
  \begin{proof}
      Using our explicit expressions above, it can easily be shown that for both models,
      \[ \lim_{n\rightarrow\infty} \left|\frac{\chi_n(x)\Delta_{L,n}(x)}{\chi_{n-1}(x)\Delta_{L,n-1}(x)} \right| = \lim_{n\rightarrow\infty} \left|\frac{X_+\circ\Upsilon_n(x)\Delta^S_{R,n}(x)}{X_+\circ\Upsilon_{n-1}(x)\Delta_{R,n-1}(x)} \right| =  \begin{cases}  1/|q|^4 & : |q|>1  \\  |q|^4 & : |q|<1 \end{cases}, \]
      which proves the convergence where the functions $\chi_n,\Delta_{L,n},\Delta_{R,n},$ and $X_+\circ\Upsilon_n$ are analytic.
  \end{proof}
\subsection{Asymptotic enumeration of models $\mD$ and $\mE$}  
\label{sec:aasym}
    
Both of the generating functions for the asymmetric models have a dominant singularity at $t=1/|\mS|$, although proving this is more complicated than in the symmetric case.  For model~\mD, the numerator of the generating function has a residue of zero (indeed a square-root singularity appears) and one must do a more careful analysis. For model~\mE, the numerator of the generating function has a non-zero residue as before, but we must consider two series in the proof.  Because of this, we simply get a bound on the growth constant at the dominant singularity -- we do not provide a mechanism for its calculation to arbitrary precision as in the symmetric cases. We make use of combinatorial arguments, so we return to the $t$-plane for the remainder of this section.

\subsubsection{Model \mD}
This case was completely considered by Mishna and Rechnitzer, and we restate their results. 

\greybox{
\begin{theorem}[Mishna and Rechnizter \cite{MiRe09}; Proposition 16]
  If $D_n$ denotes the number of walks with steps from \mD~and staying
  in the positive quarter plane, then $D_n \sim \kappa_D
  \frac{3^n}{\sqrt{n}}(1+o(1))$, where $\kappa_D \in
  \left[0,\sqrt{\frac{3}{\pi}}\right]$.
\end{theorem}
}

\subsubsection{Model \mE}
In this case, we separately consider the two generating functions of
walks returning to the axis, and bound their convergence at the point
$t=1/4$.
\begin{lemma}\label{lem:E-bound}
The function $E_{1,0}(t)$ is analytic for $|t| \leq \frac{1}{2\sqrt{3}}$, while the function $E_{0,1}(t)$ is analytic for $|t| \leq \frac{1}{1+2\sqrt{2}}$. 
\end{lemma}
\begin{proof} We use the same approach as in
  Theorem~\ref{thm:dom-sing}, bounding the exponential growth factor
  by considering walks in the half plane that end at the $x$ and $y$
  axis respectively. This proves the coefficient of $t^n$ in $E_{1,0}(t)$
  has growth bounded above by $O\left((2\sqrt{3})^n\right)$ and
  the coefficient of $t^n$ in $E_{0,1}(t)$ has growth bounded above by
  $O\left((1+2\sqrt{2})^n\right)$.  The exponential growth in the
  asymptotic expression corresponds to the inverse of the dominant
  singularity, and the result follows.
\end{proof}

\begin{lemma} \label{thm:E-sing} The function $E(t)$ has a simple
  singularity at $t=1/4$ where it has
  a residue of value $\kappa_E \in
  \left[\frac{122}{525},\frac{7}{10}\right]$.
\end{lemma} 

\begin{proof}
  To compute the residue, it suffices to
  substitute the required value into the explicit expressions for
  $\chi_n,Y_+\circ\chi_n, \Upsilon_n,$ and $X_+\circ\Upsilon_n$ for all $n$.  We
  treat the generating functions returning to the axes separately,
  with convergence established by the ratio test.  Furthermore, we can
  tightly bound the series in the numerator using the values of some
  initial terms and two telescoping series, and compute that, as desired:
\[ 1-\frac14E_{1,0}(1/4)-\frac14E_{0,1}(1/4) \in \left[\frac{122}{525},\frac{7}{10}\right] \subset [0.232,0.7].
\]
\end{proof}

    One may note that the location of the singularity is predicted, but not proven,
  by the results of~\cite{FaRa12}.  Lemmas~\ref{lem:E-bound} and~\ref{thm:E-sing} combine to give us the leading term asymptotics of Model~\mE.

\greybox{
    \begin{cor} The number, $E_n$, of walks taking steps in \mE~and staying in the positive quarter plane grows asymptotically as
    \[ E_n = \kappa_E\cdot 4^n + O\left( (1+2\sqrt{2})^n \right),\]
    where $\kappa_E \in \left[\frac{122}{525},\frac{7}{10}\right] \subset [0.232,0.7]$.
    \end{cor}
  } 

  Computational evidence given by calculating the series for
  $E_{1,0}(1/4)$ and $E_{0,1}(1/4)$ to a large number of terms implies
  that the value of the growth constant is approximately $0.2636$,
  which is consistent with the growth of computationally generated
  values of $E_n$ for large $n$.

  \subsection{The generating functions $D(t)$ and $E(t)$ are not D-finite}
  \label{sec:anonD}
  The additional sums that arise in our expressions for~$D(t)$
  and~$E(t)$ do not change our fundamental argument. In the symmetric
  examples we found a set of singularities associated to each $Y_n$
  and proved that they do not cancel. Here, although the same
  structure is undoubtedly present, we prove the existence and non 
  cancellation of singularities along a
  single line. The set of singularities is infinite, and thus the
  generating functions are not D-finite. Indeed, this argument is
  simpler and we would have emulated it in the symmetric cases had we
  found a ray or line which contained an infinite number of
  singularities.
  
  More specifically, for both asymmetric models, we demonstrate an
  infinite source of singularities in $L(1,q)$ and prove that $R_1q)$
  converges at those points.  As in the previous cases, we find
  polynomials $\omega^1_n, \omega^2_n, \omega^3_n,$ and $\omega^4_n$
  that the roots of $\overline{\chi}_n, \overline{Y_+\circ\chi}_n,
  \overline{\Upsilon}_n,$ and $\overline{X_+\circ\Upsilon}_n $ must satisfy --
  note that the orders of the roots of the polynomials match the
  orders of the roots of our functions.  These polynomials are
  summarized in Tables~\ref{tab:D-polys} and~\ref{tab:E-polys}. 
  To be more precise, $\omega^1_n(q)$ contains the poles of
  $\chi_n(q)$, $\omega^2_n(q)$ contains the poles of $Y_+ \circ \chi_n(q)$,
  etc.

\begin{table}\small
  \[
 \begin{array}{l|l}
    \chi_n& \omega^1_n = \left(q^{4n+2}+q^{2n+4}-4q^{2n+2}+q^{2n}+q^2\right)^2\\[2mm]
     Y_+\circ\chi_n& \omega^2_n = \left(q^{4n+2}+q^{2n+4}-4q^{2n+2}+q^{2n}+q^2\right) \left(q^{4n+4}+q^{2n+4}-4q^{2n+2}+q^{2n}+1\right)\\[2mm]
     \Upsilon_n& \omega^3_n=
     \left(q^{4n+3}+q^{2n+4}-q^{2n+3}-2q^{2n+2}-q^{2n+1}+q^{2n}+q\right)\\
&\hfill
 \cdot\left(q^{4n+1}+q^{2n+4}-q^{2n+3}-2q^{2n+2}-q^{2n+1}+q^{2n}+q^3\right)\\[2mm]
      X_+\circ\Upsilon_n& \omega^3_n = \left(  q^{4n+3}+q^{2n+4}-q^{2n+3}-2q^{2n+2}-q^{2n+1}+q^{2n}+q \right)^2\\
    \end{array}
\]
\caption{The minimal polynomials of the singularities for functions
  defined by Eqn.~\eqref{eq:closedform} (Model \mD).}
\label{tab:D-polys}
\end{table}

\begin{table}[ht]\small
\[
  \begin{array}{l|ll}
  \chi_n&    \omega^1_n = q^2\left(q^4-q^2+1\right)\left(q^{8n}+1\right)+2q^2(q^4-4q^2+1)\left(q^{6n}+q^{2n}\right)+(q^8-10q^6+24q^4-10q^2+1)q^{4n}\\[2mm]
    Y_+\circ\chi_n&  \omega^2_n = (q^4-q^2+1)\left(q^{8n+4}+1\right)+(q^6-3q^4-3q^2+1)\left(q^{6n+2}+q^{2n}\right)+(q^8-9q^6+22q^4-9q^2+1)q^{4n}\\[2mm]
      \Upsilon_n&\omega^3_n = q^2(q^4-q^2+1)\left(q^{8n}+1\right)+q(q^6-q^5-q^4-2q^3-q^2-q+1)\left(q^{6n}+q^{2n}\right)\\
                               &\hfill+(q^8-2q^7-4q^6+2q^5+12q^4+2q^3-4q^2-2q+1)q^{4n} \\[2mm]
      X_+\circ\Upsilon_n&\omega^4_n = (q^4-q^2+1)\left(q^{8n+4}+1\right)+2q(q^4-q^3-2q^2-q+1)\left(q^{6n+2}+q^{2n}\right)\\
                          &\hfill+(q^8-2q^7-5q^6+2q^5+14q^4+2q^3-5q^2-2q+1)q^{4n}\\
  \end{array}
\]
\caption{The minimal polynomials of the singularities for functions
  defined by Eqn.~\eqref{eq:closedform} (Model \mE).}
\label{tab:E-polys}
\end{table}
We next prove that for even $n$, each $\chi_n$ has a distinct
singularity on the imaginary axis, and we prove that it is indeed a
singularity of the generating function. We prove this separately for
each model, but the arguments (indeed the computations!) are almost
identical in both cases. In order to manipulate the unwieldy formulas
which arise we used the Groebner package in Maple version 16 to
calculate the relevant Gr\"obner bases.

Both cases also invoke $\chi_{-n}(q)$ to
prove that certain solutions of the polynomial are actually solutions
of the model. These are defined by rolling the recurrence in reverse,
as before. 
\subsection{Model \mD}
First, note that in the case of model~\mD~the poles of $\chi_n(q)$ are contained in the roots of 
$\omega^1_n = \left(q^{4n+2}+q^{2n+4}-4q^{2n+2}+q^{2n}+q^2\right)^2$, by Table~\ref{tab:D-polys}. 
\begin{lemma} The function $\overline{\chi}_n(q)$ has a root on the imaginary axis between $i$ and $2i$, when $n$ is even.
\end{lemma}
  \begin{proof}
    Suppose $r\in\mathbb{R}$ and substitute~$q=ri$ into $\omega^1_n(q)$:
    \[ \omega^1_n(ri) = R^2-r^2+4r^2R^2-r^2R^4+r^4R^2,\]
    where~$R=r^n$.  We remark that this is a real valued function of $r$, and if $r=1$ then $R=1$ and $\omega^1_n(i)=4$.  However, $\omega^1_n(2i)=33R^2-4-4R^4$ which is negative for $R\geq 2$. Thus, the Intermediate Value Theorem implies $\omega^1_n(ri)$ has a zero on the imaginary axis between $i$ and $2i$. Denote this value by $r_c$.  The expression for $\omega^1_n(ri)$ is palindromic, so $r=1/r_c$ is also a root of $\omega^1_n(ri)$.

 For $r>0$ it can easily be shown that $\chi_n\left(i/r\right) = \chi_{-n}(ri)$, and thus one of $ir_c$ and $i/r_c$ must be a root of $\overline{\chi}_n$.    As the numerator of $\overline{\chi}_n(ri)$ is 
    \[ 4R^4r^2+(r^4+1)(1+R^4)+(r^2+1)(1-R^4)\sqrt{1+6r^2+r^4}+4r^2(1-R^2) ,\]
    which is strictly positive for $0<r<1$, the root of $\omega^1_n$ between $i$ and $2i$ is in fact a root of $\overline{\chi}_n$.
  \end{proof}

Furthermore, the poles of $\chi_n$ and the poles of $\chi_k$ are distinct when $n\neq k$.  To see this, we note that we can re-write the expression above for $\omega^1_n(ri)$ as $-R^4r^2+(1+4r^2+r^4)R^2-r^2$.  Treating $r$ and $R$ as independent variables, for a fixed positive value of $r$ this expression is a polynomial in $R$ whose coefficients have signs negative, positive, negative, respectively.  Descartes's Rule of Signs then implies that there are at most 2 roots to this equation, one of which we know to be inside the unit circle and one of which lies outside of it.  The result then follows from the observation that different values of $n$ each yield a different value of $(ri)^n=R^n$ for $r>1$.

Next we show that these singularities are present in the term $\chi_n \Delta_{L, n}(q)$.  To do this, we let $o\geq1$ denote the multiplicity of the root $q=ri$ of $q^{4n+2}+q^{2n+4}-4q^{2n+2}+q^{2n}+q^2$ we have found above, so that  $\chi_n(q)$ has a pole of order $2o$ at $q=ri$.

\begin{lemma}  \label{lem:Ddistinctsing} 
  For an infinite number of $n$ there exists a distinct
  purely imaginary number $ri$ with $1<r<2$, such that
  $\chi_n\left(Y_+\circ\chi_n-Y_+\circ\chi_{n-1}\right)$ has a pole
  (of order $3o\geq3$) at $q=ri$.  Furthermore, for $k \neq n$ the
  summand $\chi_k\left(Y_+\circ\chi_k-Y_+\circ\chi_{k-1}\right)$ of
  $L(1,q)$ is analytic at $ri$.\end{lemma}

\begin{proof} Let $ri$ be the root of $\overline{\chi}_n$ described
  above.  Then using the identity
     \[ Y_+\circ\chi_n - Y_+\circ\chi_{n-1} = \frac{\overline{Y_+\circ\chi}_{n-1} - \overline{Y_+\circ\chi}_n}{\overline{Y_+\circ\chi}_{n-1}\overline{Y_+\circ\chi}_n}, \]
     we find that the zeroes of $Y_+\circ\chi_n - Y_+\circ\chi_{n-1} $ are also roots (with the same multiplicity) of the polynomial equation

{\small      \begin{equation*}
\left(q^{2n+1}+q^{n+2}-q^n+q\right)
\left(q^{2n+1}+q^n+q-q^{n+2}\right)
\left(q^{4n+2}+q^{2n+4}-4q^{2n+2}+q^{2n}+q^2\right)=0,
\end{equation*}}

     while the poles of $Y_+\circ\chi_n - Y_+\circ\chi_{n-1} $ are also roots (with
     the same multiplicity) of the polynomial equation
{\small      \begin{equation*}
  \left(q^{4n+4}+q^{2n+4}-4q^{2n+2}+q^{2n}+1\right)
  \left(q^{4n}+q^{2n+4}-4q^{2n+2}+q^{2n}+q^4\right)
  \left(q^{4n+2}+q^{2n+4}-4q^{2n+2}+q^{2n}+q^2\right)^2=0.
\end{equation*}}
     Note that the last factor of each of these polynomials is the
     factor that appears in $\omega^1_n$!  In fact, by treating $q^n$
     as an independent variable $Q$ and taking a Gr\"obner basis with
     respect to a lexicographical ordering one can show that for all
     values of $q$ not satisfying $(q^4-4q^2+1)(q^2+3q+1)(q^2-3q+1)=0$
     these are the only factors of the above polynomials that can
     share roots with $\omega^1_n$.  As
     $(q^4-4q^2+1)(q^2+3q+1)(q^2-3q+1)$ has no zeroes on the imaginary
     axis, we can see that $Y_+\circ\chi_n-Y_+\circ\chi_{n-1}$ has a pole of order
     $o$ at $ri$, which combines with the pole of $\chi_n$, which has
     order $2o$, to yield the result.

     We have already noted above that $\chi_k$ shares no poles with
     $\chi_n$ when $k \neq n$, so it is sufficient to show
     $Y_+\circ\chi_k-Y_+\circ\chi_{k-1}$ does not have a pole at $ri$.  Indeed,
     since the poles of $Y_+\circ\chi_k-Y_+\circ\chi_{k-1}$ are roots of
     \[\left(q^{4k+4}+q^{2k+4}-4q^{2k+2}+q^{2k}+1\right)   \left(q^{4k}+q^{2k+4}-4q^{2k+2}+q^{2k}+q^4\right)  \left(q^{4k+2}+q^{2k+4}-4q^{2k+2}+q^{2k}+q^2\right)^2,\]
     we can substitute $q=ri$, set $P = r^k$, and factor the result to
     see that the only possible way this polynomial can be zero for
     $|r|>1$ is if $rP^2-P-r-Pr^2=0$.  Taking a Gr\"obner basis (again
     with respect to a lexicographical ordering) of this polynomial
     with the one obtained by substituting $R=r^n$ in
     $\omega^1_n(ri)$, we get $(PR-1)(PR+1)(R-P)(R+P)$ as a generator,
     which cannot equal 0 when $|r|>1$.
   \end{proof}

\greybox{
\begin{theorem}
The generating function~$D(t)$ for walks in the quarter plane with steps from~$\mD$ is not D-finite. 
\end{theorem}
}
\begin{proof} First, recall that $D(q/(1+q^2)) = \frac{(1+q^2) -
    L(1,q) - R(1,q)}{q^2-3q+1}.$ The proof is similar to the symmetric
  case, except we restrict our attention to poles located on the
  imaginary axis. For each~$n$, Lemma~\ref{lem:Ddistinctsing}
  describes $q_n$, a purely imaginary pole of the $n$-th summand in
  $L(1,q)$. We show that it is also a pole of $L(1,q)$, and then
  verify that it is not a pole of $R(1,q)$. Consequently it is not
  cancelled in the full expression, and is also a pole of
  $D(q/(1+q^2))$.

  Lemma~\ref{lem:Ddistinctsing} proves that the purely imaginary poles
  of the term $\chi_n\left(Y_+\circ\chi_n-Y_+\circ\chi_{n-1}\right)$
  are not poles of any other term of the main summation in our
  expression for $D_{1,0}\left(\frac{q}{1+q^2}\right)$. The arguments
  of the symmetric case can be used almost verbatim to show that the
  poles are also poles of $L(1,q)$.

  Next, we show that none of the summands of $R(1,q)$ have purely
  imaginary poles. We consider the family of polynomials that the poles satisfy,
  and show that none of them have purely imaginary roots via a
  Gr\"obner basis computation, where we replace the $q^n$ terms by a
  single variable in order to make the computation generic. 

  Specifically, the computation is done as follows: 
  substitute $q=ri$ into $\omega^3_n$, set $R=r^n$, and take a
  Gr\"obner basis of the real and imaginary parts of the resulting
  polynomial with respect to a lexicographical ordering.  The result
  has $r^5(r-1)(r+1)(r^2+1)^6$ as an element, so $\omega^3_n$ has no
  root of the form $q=ri$ with $r>1$.  Similarly, the analogous
  Gr\"obner basis computation on the real and imaginary parts of
  $\omega^4_n(ri)$ has $r^5(r^2+1)^6$ as an element.  Thus,
  $\Upsilon^S_n$ and $X_+\circ\Upsilon^S_n$ have no poles of the form
  $q=ri$ with $r>1$, so $R(1,q)$ contains no pole located at any of the singularities of $L(1,q)$ 
  described above.

  By now we are almost on autopilot: any singularity of $L(1,q)$ not
  cancelled by $R(1,q)$ is a singularity of the complete
  expression. Since we have an infinite family of poles,
  $D(1/(1+q^2))$ is not D-finite since it has an infinite family of
  singularities on the imaginary axis. As a consequence, $D(t)$ is
  also not D-finite since it can be obtained from this function by
  algebraic substitution. 
   \end{proof}
 
 \subsubsection{Model $\mE$}
The argument to show that this model is not D-finite is identical, save for the actual location of the singularities on the imaginary axis. We highlight the differences. 
\begin{lemma} 
The function $\overline{\chi}^E_n(q)$ has a root on the imaginary axis between $i$ and $2i$. 
\end{lemma}
\begin{proof}
    As before, we substitute $q=ri$ into the equation for $\overline{\chi}_n(q)$ to get
    \[ \overline{\chi}_n(ri) = 4R^4r^2+4r^2-4R^2r^2+R^4r^4+R^4+r^4+1+(r^2+1-R^4-R^4r^2)\sqrt{1+10r^2+r^4},\]
    where $R=r^n$.  If we substitute $r=R=1$ into this expression, we get the value 8.  If we substitute $r=2$, we get the expression $33R^4+33-16R^2+(5-5R^4)\sqrt{57}$ which is negative for $R\geq2$.  Thus, the Intermediate Value Theorem implies that $\overline{\chi}_n(ri)$ has a root for $r \in (1,2)$.
    \end{proof}
As was the case with Model~\mD, Descartes's Rule of Signs allows us to
conclude this root is unique as
{\small\begin{equation*}
      \overline{\chi}_n(ri) =
    -r^2(1+r^2+r^4)R^8-2r^2(1+4r^2+r^4)R^6+(1+10r^2+24r^4+10r^6+r^8)R^4-2r^2(1+4r^2+r^4)R^2-r^2(1+r^2+r^4) 
\end{equation*}
}
    has two sign changes in its coefficients when viewed as a polynomial in $R$.  Now, we prove that the other functions under consideration have no imaginary poles.
    \begin{lemma} The functions $Y_+\circ\chi_n(q)$, $\Upsilon_n(q)$, and $X_+\circ\Upsilon_n(q)$ have no poles $q$ on the imaginary axis. \end{lemma}
    \begin{proof} Substituting $q=ri$ into our expression for $\omega^2_n(q)$ gives
    \begin{align*}
    \omega^2_n(ri) &=  r^8R^8+r^4+r^6R^8+r^2+r^4R^8+1+R^2+3R^2r^2+R^4+9R^4r^2+22R^4r^4 \\
                                       &+ 3R^6r^4(r^2-1) + 3R^2r^4(3R^2r^2-1) + R^6r^2(r^6-1) + R^2r^6(R^2r^2-1),
     \end{align*}
     which is strictly positive for $r>1$.  To prove that $\overline{\Upsilon}_n(q)$ and $\overline{X_+\circ\Upsilon}_n(q)$ have no roots on the imaginary axis, we substitute $q=ri$ into $\omega^3_n(q)$ and $\omega^4_n(q)$ obtaining expressions which have real and imaginary components which are non-zero polynomials in $r$ and $R$.  Taking a Gr\"obner basis of these real and imaginary components allows us to eliminate $R$ in each case and prove the result. \end{proof}

This has the immediate corollary that $R(1,q)$ admits no polar singularities at these poles of $\chi_n(q)$, of which we have found an infinite number.  Next, we show the poles do not cancel.

\begin{lemma} At each of the poles found above, $Y_+\circ\chi_n(q) - Y_+\circ\chi_{n-1}(q) \neq 0.$ \end{lemma}
\begin{proof}  Substituting $q=ri$ into our explicit expression for $\overline{Y_+\circ\chi}_n(q)$ allows us to determine that any root of $Y_+\circ\chi_n(q) - Y_+\circ\chi_{n-1}(q)$ must satisfy the polynomial equation
    \[ (4r^6+9r^4+4r^2)R^8-(r^8+1)R^4+4r^6+9r^4+4r^2=0. \]
    The Gr\"obner Basis of this polynomial and the one we found for $\overline{\chi}_n(ri)$ with respect to a lexicographical ordering has
    \[ r^2(9r^{16}+236r^{14}+2148r^{12}+7684r^{10}+11974r^8+7684r^6+2148r^4+236r^2+9)(1+r^2)^4 \]
    as one of its generators.  This has no positive real roots in~$r$, so the result holds.
    \end{proof}
This allows us to conclude our study of Model~\mE~with the
    following theorem.

\greybox{
\begin{theorem}
The generating function~$E(t)$ for walks in the quarter plane with steps from $\mE$ is not D-finite. 
\end{theorem}
}

\section{Conclusion}
This work addresses a family of lattice path models that have resisted
other powerful approaches.  There could also be other models, with
larger step sizes or in higher dimensions, to which this method is
suitable. In three dimensions the challenge is to set up the equations
in such a way that unknowns are canceled at the same rate in which
they are generated.  Ideally, we would like to automate as much as is
possible.

In this model, the connection between the infiniteness of the group
and the infinite number of singularities is quite transparent. Is
there any hope to transport this concept to the remaining small step
models in the quarter plane, and show that their counting functions
are also not D-finite (in addition to the provably non-D-finite
multivariate generatingfunction)?

\section{Acknowledgments}
We are indebted to Alin Bostan, Mireille Bousquet-M{\'e}lou, Manuel Kauers, 
Pierre Lairez, Kilian Raschel, and Andrew Rechnitzer for some key
discussions and insights. We are also grateful to the referees for their thoughtful comments and remarks. 


\end{document}